\documentclass[preprint,12pt]{elsarticle}

\usepackage{amssymb}
\usepackage{amsmath}

\usepackage[noend,ruled]{algorithm2e} 
\usepackage{float}
\usepackage{mathrsfs}
\usepackage{cleveref}
\usepackage{etoolbox}
\usepackage{xurl}

\newcommand{\N}{\mathbb N}
\newcommand{\Q}{\mathbb Q}
\newcommand{\bA}{{\mathfrak A}}
\newcommand{\bB}{{\mathfrak B}}
\newcommand{\bC}{{\mathfrak C}}

\newcommand{\bF}{{\mathfrak F}}

\newcommand{\checked}[1]{#1}

\DeclareMathOperator{\Cyc}{Cycl}

\DeclareMathOperator{\Aut}{Aut}
\DeclareMathOperator{\End}{End}

\DeclareMathOperator{\fPol}{fPol}
\DeclareMathOperator{\VCSP}{VCSP}

\DeclareMathOperator{\vcsp}{VCSP}
\DeclareMathOperator{\CSP}{CSP}

\DeclareMathOperator{\Pol}{Pol}

\DeclareMathOperator{\Opt}{Opt}
\DeclareMathOperator{\Id}{Id}
\DeclareMathOperator{\Feas}{Feas}
\DeclareMathOperator{\Sym}{Sym}

\DeclareMathOperator{\OIT}{OIT}
\DeclareMathOperator{\mi}{mi}
\DeclareMathOperator{\lele}{ll}
\DeclareMathOperator{\mx}{mx}
\DeclareMathOperator{\lex}{lex}
\DeclareMathOperator{\mix}{mix}

\DeclareMathOperator{\inj}{inj}
\DeclareMathOperator{\const}{const}

\DeclareMathOperator{\Betw}{Betw}
\DeclareMathOperator{\Sep}{Sep}
\DeclareMathOperator{\Dis}{Dis}

\newtheorem{theorem}{Theorem}
\newtheorem{lemma}[theorem]{Lemma}
\newtheorem{proposition}[theorem]{Proposition}
\newtheorem{corollary}[theorem]{Corollary}
\newtheorem{observation}[theorem]{Observation}
\newtheorem{definition}[theorem]{Definition}
\newtheorem{example}[theorem]{Example}
\newtheorem{question}[theorem]{Question}
\newtheorem{remark}[theorem]{Remark}

\newproof{proof}{Proof}
\AtEndEnvironment{proof}{\qed}

\journal{Information and Computation}

\begin{document}

\begin{frontmatter}



\title{A Complexity Dichotomy for Temporal Valued Constraint Satisfaction Problems}


\author[MB]{Manuel Bodirsky} 
\author[EB]{\'Edouard Bonnet} 
\author[ZS]{\v{Z}aneta Semani\v{s}inov\'{a}} 

\affiliation[MB]{organization={Institute of Algebra, TU Dresden},
            city={Dresden},
            postcode={01062}, 
            country={Germany}}

\affiliation[EB]{organization={Univ Lyon, CNRS, ENS de Lyon, Université Claude Bernard Lyon 1, LIP UMR5668},
            city={Lyon},
            postcode={69342}, 
            country={France}}

\affiliation[ZS]{organization={Institute of Discrete Mathematics and Geometry, TU Wien},
            addressline={Karlsplatz 13}, 
            city={Wien},
            postcode={1040}, 
            country={Austria}}

\begin{abstract}
We study the computational complexity of the \emph{valued} constraint satisfaction problem (VCSP) for every valued structure over ${\mathbb Q}$ that is preserved by all order-preserving bijections. Such VCSPs will be called \emph{temporal}, in analogy to the (classical) constraint satisfaction problem: a~relational structure is preserved by all order-preserving bijections if and only if all its relations have a~first-order definition in $({\mathbb Q};<)$, and the CSPs for such structures are called \emph{temporal} CSPs. 
Many optimization problems that have been studied intensively in the literature can be phrased as a~temporal VCSP. 
We prove that a~temporal VCSP is in P, or NP-complete. Our analysis uses the concept of \emph{fractional polymorphisms}.  
This is the first dichotomy result for VCSPs over infinite domains which is complete in the sense that it treats \emph{all} valued structures that contain a given automorphism group.
\end{abstract}




\end{frontmatter}




\section{Introduction}
\emph{Valued constraint satisfaction problems (VCSPs)} form
a large class of computational optimization problems.
A VCSP is parameterized by a \emph{valued structure} (called the \emph{template}), which consists of a \emph{domain} $D$ and cost functions, each defined on $D^k$ for some $k$. The input to the VCSP consists of a finite set of variables, a finite sum of cost functions applied to these variables,  and a threshold $u$, 
and the task is to find an assignment to the variables so that the sum of the costs is at most $u$. The computational complexity of such problems has been studied depending on
the valued structure that parameterizes the problem. 
VCSPs generalize constraint satisfaction problems (CSPs), which can be viewed as a variant of VCSPs with costs from the set $\{0, \infty\}$: every constraint is either satisfied or surpasses every finite threshold. 
VCSPs also generalize min-CSPs, 
which are the natural variant of CSPs
where, instead of asking whether all constraints can be satisfied at once, we search for an assignment that minimizes the number of unsatisfied constraints. Such problems can be modeled as VCSPs with costs from the set $\{0,1\}$.

A major achievement of the field is that if the domain
of the valued structure $\bA$ is finite, then the computational complexity of $\vcsp(\bA)$ is in P, or NP-complete. This result has an interesting history.
The classification task was first considered in~\cite{CohenCooperJeavonsVCSP} with important first results that indicated that we might expect a good systematic theory for such VCSPs. A milestone was reached by Thapper and \v{Z}ivn\'{y} with the proof of a complexity dichotomy for the case where the cost functions never take value $\infty$~\cite{ThapperZivny13}.
On the hardness side, 
Kozik and Ochremiak~\cite{KozikOchremiak15} formulated a condition that implies hardness for $\vcsp(\bA)$ and found equivalent characterisations that suggested that this condition characterises NP-hardness (unless P=NP, of course). Kolmogorov, Krokhin, and Rol\'{i}nek~\cite{KolmogorovKR17} then showed that if the hardness condition from~\cite{KozikOchremiak15} does not apply, linear programming relaxation in combination with algorithms for classical CSPs can be used to solve $\vcsp(\bA)$, conditional on the tractability conjecture for (classical) CSPs. Finally, this conjecture about CSPs has been confirmed
independently by Bulatov~\cite{BulatovFVConjecture} and by Zhuk~\cite{ZhukFVConjecture, Zhuk20}, thus completing the complexity dichotomy for $\vcsp(\bA)$ for finite-domain templates $\bA$ as well.
A key tool for distinguishing the tractable VCSPs from the NP-hard ones are \emph{fractional polymorphisms}, which, in some sense, capture the symmetries of the VCSP template. 

Many important optimization problems in the literature cannot be modeled as VCSPs if we restrict to valued structures on a finite domain;
VCSPs that require an infinite domain are, for example, the min-correlation-clustering problem with partial information~\cite{CorrelationClustering,ViolaThesis}, ordering min-CSPs~\cite{ApproxOrderingCSP}, phylogeny min-CSPs~\cite{ChatziafratisM23}, VCSPs with semilinear constraints~\cite{BodirskyMaminoViola-Journal}, and the class of resilience problems from database theory~\cite{Resilience, NewResilience, LatestResilience, Resilience-VCSPs}.

For VCSPs with infinite templates we cannot hope for general classification results, since this is already out of reach for the special case of CSPs over infinite domains \cite{BodirskyGrohe}. However, a powerful algebraic machinery was developed to study CSPs of structures with a rich automorphism group,
which has led to
classification results for many concrete automorphism groups: we list \cite{tcsps-journal,BodPin-Schaefer-both,phylo-long,BBSpatial,AnOrderOutOf}
as a representative sample. Some part of this machinery has also been developed for VCSPs in \cite{Resilience-VCSPs}, inspired by the concepts from infinite-domain CSPs and finite-domain VCSPs~\cite{FullaZivny,KozikOchremiak15}. 
Nevertheless, no complexity classification for VCSPs such as for the classes of CSPs discussed above was obtained so far.

In this article we provide the first VCSP dichotomy result for a class of templates 
which consists of \emph{all} valued structures preserved by some fixed permutation group.
Concretely, we prove that the VCSP for every valued structure with the domain $\Q$ that is preserved by all order-preserving bijections is in P or NP-complete. We call such valued structures \emph{temporal}, in analogy to temporal relational structures, i.e., structures with the domain  $\Q$ preserved by all order-preserving bijections -- these are precisely the structures with a 
first-order definition over $(\Q;<)$. We also provide disjoint algebraic conditions that characterize the tractable and the NP-complete case, in analogy to similar classifications for classes of (V)CSPs. The result confirms the dichotomy conjecture from \cite[Conjecture 9.3]{Resilience-VCSPs} for the special case of temporal valued structures. Note that our classification result is incomparable with the result of~\cite{BodirskyMaminoViola-Journal} which shows that that submodular PLH functions form a maximally tractable class of PLH cost functions, because the class of temporal valued structures is a proper subclass of PLH and contains structures with tractable VCSPs that are not submodular.

Apart from the relevance of temporal VCSPs as a test case for understanding VCSPs on infinite domains, 
they constitute an important class for several reasons:
\begin{itemize}
\item 
Temporal VCSPs encompass many natural optimization problems\footnote{It should be mentioned that these problems come in two flavours: one is where the input is a graph, or more generally a structure; the other, which we adopt here, is that the input consists of a finite sum of cost functions, which in particular allows that the sum contains identical summands. In the setting of Min-CSPs for graphs, this corresponds to considering \emph{multigraphs} in  the input. However, often (but not always, see Theorems 8.6 and 8.7 in~\cite{LatestResilience}) the two variants of the problem have the same complexity (see, e.g.,~\cite[Section 2.3]{Interval} for relevant techniques in this context).}, for example, \emph{Directed Feedback Arc Set}, \emph{Directed Subset Feedback Arc Set}, \emph{Edge Multicut} (aka Min-Correlation-Clustering with Partial Information), 
\emph{Symmetric Directed Multicut},
\emph{Steiner Multicut}, \emph{Disjunctive Multicut}, and many more~\cite{Multicut,EibenRW22,OsipovPilipczuk,OsipovW23,PointAlgebraMinCSP}.
\item The complete classification for temporal CSPs has been the basis to obtain other complexity classifications in temporal and spatial reasoning via complexity classification transfer techniques~\cite{Products}; we expect that temporal valued CSPs play a similar role for the corresponding optimisation problems.
\item Temporal VCSPs contain interesting polynomial-time tractable cases with some non-trivial interaction between the `soft' (i.e., finite-valued) and `crisp' constraints (i.e., $\{0,\infty\}$-valued) in the input ( Lemma~\ref{lem:lex-algo}). 
\end{itemize}

Our tractability condition for temporal VCSPs is phrased in terms of  \emph{fractional polymorphisms}, that is, probability distributions on operations with particular properties. Surprisingly, the fractional polymorphisms that appear in this classification are of a very particular shape: there is always a single operation with probability 1. 
It is known that in general, such fractional polymorphisms are not sufficient to capture the border between polynomial-time tractable and NP-hard VCSPs~\cite{cohen2006complexity}, and we therefore present fractional polymorphisms in full generality for consistency with the literature. 
The concrete operations that appear in this context are the same operations that were already essential for the the classification of temporal CSPs. 
The hardness condition is based on the notion of \emph{\checked{pp-}expressibility}, which generalizes primitive positive definitions, and on the notion of generalized \emph{pp-constructions}~\cite{Resilience-VCSPs}, which provide polynomial-time reductions between VCSPs. Interestingly enough, it is not known whether preservation by fractional polymorphisms characterises \checked{pp-}expressibility in our setting; this is known for valued structures with a finite domain~\cite{VCSP-Galois,FullaZivny}.
Our classification proof, however, does not rely on such a characterisation.

\subsection{Related work} 
VCSPs on infinite domains have been studied in~\cite{ViolaThesis, SchneiderViola, ViolaZivny}. However, the valued structures considered in these articles typically do not have an oligomorphic automorphism group, a property that is essential for applying our techniques for classifying the complexity of VCSPs. The foundations of the theory for VCSPs of valued structures with an oligomorphic automorphism group were laid in~\cite{Resilience-VCSPs} with the motivation to study \emph{resilience problems} from database theory. Concrete subclasses of temporal VCSPs have been studied in the context of min-CSPs from the parameterized complexity perspective: the complexity of \emph{equality min-CSPs} has been classified in~\cite{equalityminCSP} and parameterized complexity of min-CSPs over the Point Algebra has been classified in~\cite{PointAlgebraMinCSP}. The authors mention a classification of first-order generalizations of the Point Algebra as a natural continuation of their research \cite[Section 5]{PointAlgebraMinCSP}. The present paper contains a classification of VCSPs over all temporal structures and thus provides
the foundations for classifying the parameterized complexity of min-CSPs for such structures, including
algebraic techniques
that we expect to be useful in this context as well.

\subsection{Outline} 
The article is organized as follows. Section~\ref{sect:prelims} contains preliminaries on VCSPs in general, with some notation and properties specific to temporal VCSPs. Section~\ref{sect:gen-facts} contains several new facts about VCSPs that have been used in the classification. Section~\ref{sect:evcsps} contains the classification of \emph{equality} VCSPs, that is, VCSPs of valued structures with an automorphism group equal to the full symmetric group; on the one hand, this serves as a warm-up, on the other hand it is a building block for the general case.  
Section~\ref{sect:temp} contains the full classification of temporal VCSPs, which is the main contribution of the paper.

\section{Preliminaries}
\label{sect:prelims}
Let ${\mathbb N} := \{0,1,2,\dots\}$ be the set of natural numbers.
For $k \in \N$ the set $\{1,\dots, k\}$ will be denoted by $[k]$. The set of rational numbers is denoted by $\Q$ and the standard strict linear order of $\Q$ by $<$.
We also need an additional value $\infty$; all
we need to know about $\infty$ is that
\begin{itemize}
\item $a < \infty$ for every $a \in {\mathbb Q}$,
\item $a + \infty = \infty + a~= \infty$ for all $a \in {\mathbb Q} \cup \{\infty\}$, and 
\item $0 \cdot \infty = 
\infty \cdot 0 = 0$
and $a \cdot \infty =
\infty \cdot a~= \infty$ for $a > 0$. 
 \end{itemize}

If $A$ is a~set, then $\Sym(A)$ denotes the group of all permutations of $A$. If  $t \in A^k$, then we implicitly assume that $t=(t_1, \dots, t_k)$, where $t_1,\dots,t_k \in A$. For an operation  $f \colon A^{\ell} \to A$ 
and $t^1, \dots,t^\ell \in A^k$, we
use the following notation for applying $f$ componentwise: 
\[ f(t^1, \dots, t^\ell) := (f(t^1_1, t^2_1, \dots, t^\ell_1), \ldots, f(t^1_k, t^2_k, \dots, t^\ell_k)).\]

\subsection{Valued structures}

Let $A$ be a~set and let $k \in {\mathbb N}$. 
A \emph{valued relation of arity $k$ over $A$}
is a~function $R \colon A^k \to {\mathbb Q} \cup \{\infty\}$.  
We write ${\mathscr R}_A^{(k)}$ for the set of all valued relations over $A$ of arity $k$, and define \[{\mathscr R}_A := \bigcup_{k \in {\mathbb N}} {\mathscr R}_A^{(k)}.\]
A valued relation is called \emph{finite-valued} if it takes values only in $\Q$.

Usual relations $R \subseteq A^k$ will also be called \emph{crisp} relations.
A~valued relation $R \in {\mathscr R}_A^{(k)}$ that only takes values from $\{0,\infty\}$ will be identified with the crisp relation 
$\{t \in A^k \mid R(t) = 0\}.$
A valued relation $R \in {\mathscr R}_A^{(k)}$ is called \emph{essentially crisp} if 
there exists $a \in \Q$ such that $R(t) \in \{a, \infty\}$ for every $t \in A^k$.
For $R \in {\mathscr R}_A^{(k)}$ the \emph{feasibility relation of $R$} is defined as
$ \Feas(R) := \{t \in A^k \mid R(t) < \infty\}.$
For $S \subseteq A^k$ and $a,b \in {\mathbb Q} \cup \{\infty\}$, we denote by $S_a^b$ the valued relation such that $S_a^b(t) = a$ if $t \in S$, and $S_a^b(t) = b$ otherwise.
We write $S_0^\infty$ to stress that $S$ is a crisp relation viewed as a valued relation. 
The unary empty relation on $A$, where every element of $A$ evaluates to $\infty$, is denoted by $\bot$.

\begin{example}
On the domain $\Q$, the valued relation $(=)_0^\infty$  denotes the crisp equality relation, while $(<)_0^1$ denotes the valued relation $(<)_0^1(x,y) = 0$ if $x < y$ and $(<)_0^1(x,y) = 1$ if $x \geq y$. 
\end{example}

A \emph{(relational) signature} $\tau$ is a~set of \emph{relation symbols}, each of them equipped with an arity from~${\mathbb N}$. A~\emph{valued $\tau$-structure} $\bA$ consists of a~set $A$, which is also called the \emph{domain} of $\bA$, and a~valued relation $R^{\bA} \in {\mathscr R}_A^{(k)}$ for each relation symbol $R \in \tau$ of arity $k$. All valued structures in this article have countable domains. We often write $R$ instead of $R^{\bA}$ if the valued structure is clear from the context. When not specified, we assume that the domains of valued structures $\bA, \bB, \bC, \dots$ are denoted $A, B, C, \dots$, respectively. 
If ${\mathcal R}$ is a~set of valued relations over a~common domain $A$, we write $(A;{\mathcal R})$ for a~valued structure $\bA$ whose relations are precisely the relations from $\mathcal R$; we only use this notation if the precise choice of the signature does not matter.
A valued $\tau$-structure where all valued relations only take values from $\{0,\infty\}$
may then be viewed as a~\emph{relational} or \emph{crisp} $\tau$-structure in the classical sense. A valued structure is called \emph{essentially crisp} if all of its valued relations are essentially crisp. If $\bA$ is a~valued $\tau$-structure on the domain $A$, then $\Feas(\bA)$ denotes the relational $\tau$-structure $\bA'$ on the domain $A$ where $R^{\bA'}=\Feas(R^\bA)$
for every $R\in \tau$. If $\sigma \subseteq \tau$ and $\bA'$ is a valued $\sigma$-structure such that $R^{\bA'} = R^{\bA}$ for every $R \in \sigma$, then we call $\bA'$ a \emph{reduct of $\bA$}.

\subsection{Valued constraint satisfaction problems}
Let $\tau$ be a~relational signature. An \emph{atomic \mbox{$\tau$-expression}} is an expression
of the form $R(x_1,\dots,x_k)$ for $R \in \tau$, $(=)_0^\infty(x_1, x_2)$, or $\bot(x_1)$
where  $x_1,\dots,x_k$ are (not necessarily distinct) variable symbols.
A \emph{$\tau$-expression} is an expression $\phi$ 
of the form 
$\sum_{i \leq m} \phi_i$
where $m \in {\mathbb N}$ 
and $\phi_i$ for $i \in \{1,\dots,m\}$ is 
an atomic $\tau$-expression.
Note that the same atomic $\tau$-expression might appear several times in the sum. 
We write $\phi(x_1,\dots,x_n)$ for a~$\tau$-expression where all the variables  
are from the set $\{x_1,\dots,x_n\}$. 
If $\bA$ is a~valued $\tau$-structure, then a~$\tau$-expression $\phi(x_1,\dots,x_n)$ defines over $\bA$ a~member of ${\mathscr R}_A^{(n)}$ in the usual way, which we denote by $\phi^{\bA}$ (for example, if $\phi(x,y)=R(x)+S(y)$ and $\bA$ is an $\{R,S\}$-structure, then $\phi^\bA(x,y) = R^\bA(x)+S^\bA(y)$ for all $x,y\in A$).
If $\phi$ is the empty sum, then $\phi^{\bA}$ is constant~$0$.

Let $\bA$ be a~valued structure over a~finite signature $\tau$.
The \emph{valued constraint satisfaction problem for $\bA$}, denoted by \emph{$\VCSP(\bA)$}, is the computational
problem to decide for a~given $\tau$-expression $\phi(x_1,\dots,x_n)$ 
and a~given $u \in {\mathbb Q}$
whether there exists
$t \in A^n$ such that $\phi^{\bA}(t) \leq u$. 
We refer to $\phi$ as an \emph{instance} of $\VCSP(\bA)$,
and to $u$ as the \emph{threshold}. We also refer to the pair $(\phi,u)$ as a~(positive or negative) \emph{instance} of $\VCSP(\bA)$.
A tuple $t \in A^n$ such that $\phi^{\bA}(t) \leq u$ is called a~\emph{solution for $(\phi,u)$}.
The \emph{cost}  of $\phi$ (with respect to $\bA$) is defined to be
$\inf_{t \in A^n} \phi^{\bA}(t).$
In some contexts, it will be beneficial to consider only a~given $\tau$-expression $\phi$ to be the input of $\VCSP(\bA)$ (rather than $\phi$ and the threshold $u$) and a~tuple $t \in A^n$ will then be called a~\emph{solution for $\phi$} if the cost of $\phi$ equals $\phi^{\bA}(t)$. In general, there might not be any solution. 
If there exists $t \in A^n$ such that $\phi^{\bA}(t) < \infty$ then $\phi$ is called \emph{satisfiable}. 
To give an example of a VCSP, note that $\VCSP(\Q; (<)_0^1)$ is the minimum feedback arc set problem for directed multigraphs.

If $\bA$ is a~relational $\tau$-structure, then $\CSP(\bA)$ is the problem of deciding satisfiability of conjunctions of atomic $\tau$-formulas in $\bA$.
Therefore, if $\bA$ is a relational structure then $\VCSP(\bA)$ and $\CSP(\bA)$ are essentially the same problem.  

\subsection{Automorphisms} \label{sect:aut}

Let $k \in {\mathbb N}$, 
let $R \in {\mathscr R}^{(k)}_A$, and let $\alpha$ be a~permutation of $A$. Then $\alpha$ \emph{preserves} $R$ if for all $t \in A^k$
we have $R(\alpha(t)) = R(t)$. 
If $\bA$ is a~valued structure with domain $A$, then 
an \emph{automorphism} of $\bA$ is a~permutation of $A$ that preserves all valued relations of $R$. 
The set of all automorphisms of $\bA$ is denoted by $\Aut(\bA)$, and forms a~group with respect to composition. If $\bB$ is a~valued structure and we write $\Aut(\bB) \subseteq \Aut(\bA)$ or $\Aut(\bB) = \Aut(\bA)$, we implicitly assume that $\bA$ and $\bB$ have the same domain.

Let $k \in {\mathbb N}$. 
An \emph{orbit of $k$-tuples} of a~permutation group $G$ on a set $A$ is a~set of the form $\{ \alpha(t) \mid \alpha \in G \}$ for some $t \in A^k$. 
A permutation group $G$ on a~countable set is called \emph{oligomorphic} if for every $k \in {\mathbb N}$ there are finitely many orbits of $k$-tuples in $G$ \cite{Oligo}. For example, $\Aut(\Q; <)$ and therefore every permutation group on $\Q$ that contains $\Aut(\Q; <)$ is oligomorphic. If $\bA$ is a relational structure with an oligomorphic automorphism group and $R \subseteq A^k$, then $R$ is first-order definable over $\bA$ if and only if $R$ is preserved by $\Aut(\bA)$ (this theorem, including the definition of first-order logic, is treated, e.g., in \cite{Book} (Theorem 4.2.9)).

Let $\bA$ be a~valued $\tau$-structure and $\bB$ a~relational structure. Suppose that $\Aut(\bB)$ is oligomorphic and $\Aut(\bB) \subseteq \Aut(\bA)$ (and hence $\Aut(\bA)$ is oligomorphic). Let $R \in \tau$ be of arity $k$. Then $R^{\bA}$ attains only finitely many values
by the oligomorphicity of $\Aut(\bA)$. Moreover, if for some $s,t \in A^k$ we have $R^{\bA}(s) \neq R^{\bA}(t)$, then $s$ and $t$ lie in a~different orbit of $\Aut(\bB)$. Therefore, for every value $a \in \Q \cup \{\infty\}$, there is a~union $U_a$ of orbits of $k$-tuples under the action of $\Aut(\bB)$ such that $R^{\bA}(t) = a$ if and only if $t \in U_a$. Since $U_a$ is preserved by $\Aut(\bB)$, it is first-order definable over $\bB$ by a~formula $\phi_a$.
Hence,  $R$ can be given by a~list of values $a$ in the range of $R$ and first-order formulas $\phi_a$ over $\bB$. Such a~collection $((R, a, \phi_a)\mid R \in \tau, \exists t \in A^k (R(t)=a))$ will be called a~\emph{first-order definition} of $\bA$ in $\bB$. Clearly, if a~valued structure $\bA$ has a~first-order definition in a~relational structure $\bB$, then $\Aut(\bB) \subseteq \Aut(\bA)$. Note that for some structures $\bB$ such as $(\Q;=)$ and $(\Q; <)$, the formulas $\phi_a$ can be chosen to be quantifier-free, and hence as disjunctions of conjunctions of atomic formulas over $\bB$ (in fact, this is the case for every homogeneous structure with a~finite relational signature). We will use first-order definitions of valued structures to be able to give valued structures as an input to decision problems (see
Proposition~\ref{prop:dec}).

\subsection{Expressive power}
We define generalizations of the concepts of \emph{primitive positive definitions} and \emph{relational clones}.
The motivation is that relations with a~primitive positive definition can be added to the structure without changing the complexity of the respective CSP. 

\begin{definition}
Let $A$ be a~set and $R,R' \in {\mathscr R}_A$.
We say that $R'$ can be obtained from $R$ by 
\begin{enumerate}
\item \emph{projecting} if $R'$ is of arity $k$, $R$ is of arity $k+n$ and for all $s \in A^k$, $R'(s) = \inf_{t \in A^n} R(s,t)$.
\item \emph{non-negative scaling} 
 if there exists $r \in {\mathbb Q}_{\geq 0}$ such that $R = r R'$;
 \item  \emph{shifting} if there exists $s \in {\mathbb Q}$ such that $R = R' + s$. 
\end{enumerate}
If $R$ is of arity $k$, then the relation that contains all minimal-value tuples of $R$ is 
\[\Opt(R) := \{t \in \Feas(R) \mid R(t) \leq R(s) \text{ for every } s \in A^k\}.\]
\end{definition}

Note that $\inf_{t \in A^n} R(s,t)$ in item 1 might be irrational or $-\infty$.
If this is the
case, then $\inf_{t \in A^n} R(s,t)$ does not express a~valued relation
because valued relations must have weights from ${\mathbb Q} \cup \{\infty\}$.
However, if $R$ is preserved by all permutations of
an oligomorphic automorphism group, then $R$ 
attains only finitely many values and therefore this is never the case.

If $S \subseteq \mathscr{R}_A$, then an \emph{atomic expression over $S$} is an atomic $\tau$-expression where $\tau=S$. We say that $S$ is \emph{closed under forming sums of atomic expressions} if 
for every $n \in \N$, $R_1, \dots, R_n \in S$ of arity $k_1, \dots, k_n$, respectively, and $i_1^j, \dots, i_{k_j}^j \in [k_j]$ for $j \in \{1, \dots, n \}$, the set $S$ also contains the valued relation $R$ of arity $k$ defined by
\[R(x_1, \dots, x_n) := \sum_{j=1}^n R_j(x_{i_1^j}, \dots, x_{i_{k_j}^j}).\]

\begin{definition}[valued relational clone]\label{def:wrelclone}
A \emph{valued relational clone (over $A$)} is 
a subset of 
${\mathscr R}_A$ 
that 
is closed under forming sums of atomic expressions,
projecting, shifting, non-negative scaling, $\Feas$, and $\Opt$; \checked{we will refer to expressions formed this way as \emph{pp-expressions}.}
For a~valued structure $\bA$ with the domain $A$,  
we write 
$\langle \bA \rangle$ for the smallest valued relational clone that contains the valued relations of $\bA$. 
If $R \in \langle \bA \rangle$, we say that $\bA$ \checked{\emph{pp-expresses}} $R$.
\end{definition}

Note that every relation which is primitively positively definable from a~set $S \subseteq {\mathscr R}_A$ of crisp relations lies in $\langle S \rangle$. Moreover, if $\bA$ is a relational structure and $R \in \langle \bA \rangle$, then $R$ is essentially crisp and $\Feas(R)$ is primitively positively definable from $\bA$; this is easily verified by induction. \checked{This justifies the use of the acronym `pp' that stands for primitive positive, because pp-expressions generalize primitive positive definitions.}

The following lemma is the main motivation for the concept of \checked{pp-}express\-ibility.

\begin{lemma}[Lemma 4.6 in~\cite{Resilience-VCSPs}] \label{lem:expr-reduce}
Let $\bA$ be a~valued structure on a~countable domain with an oligomorphic 
automorphism group and a~finite signature. 
Suppose that $\bB$ is a~valued structure with a~finite signature over the same domain $A$ such that every valued relation of $\bB$ is from $\langle \bA \rangle$. 
Then there is a~polynomial-time reduction from
$\VCSP(\bB)$ to $\VCSP(\bA)$.
\end{lemma}

We now introduce notation that enables us to talk about the crisp relations \checked{pp-}expressible in a valued structure, which turn out to be essential to understanding temporal VCSPs.
\begin{definition}
Let $\bA$ be a~valued structure. Then $\langle \bA \rangle_0^\infty$ denotes the set of valued relations 
$\{R \in \langle \bA \rangle \mid R  \text{ of arity }k, \forall a~\in A^k \colon R(a) \in \{0,\infty\} \}.$
\end{definition}
In words, $\langle \bA \rangle_0^\infty$ contains all crisp relations that can be \checked{pp-}expressed in $\bA$.
If $\bA$ is essentially crisp, then 
$\langle \bA \rangle = \langle \Feas(\bA) \rangle$
and $\langle \bA \rangle_0^\infty$
consists of precisely those relations that are primitively positively definable in $\Feas(\bA)$. In this case we obtain a~polynomial-time reduction from $\VCSP(\bA)$ to $\CSP(\Feas(\bA))$ and vice versa by Lemma~\ref{lem:expr-reduce}.

\subsection{Pp-constructions}
Next, we introduce a~concept of pp-constructions which give rise to poly\-no\-mial-time reductions between VCSPs. 

\begin{definition}[pp-power]
Let $\bA$ be a~valued structure with domain $A$ and let $d \in {\mathbb N}$. Then a~($d$-th) \emph{pp-power}
of $\bA$ is a~valued structure $\bB$ with domain 
$A^d$ such that for every valued relation $R$ of $\bB$ of arity $k$ there exists a~valued relation $S$ of arity $kd$ in $\langle \bA \rangle$ such that 
\[R((a^1_1,\dots,a^1_d),\dots,(a^k_1,\dots,a^k_d)) = S(a^1_1,\dots,a^1_d,\dots,a^k_1,\dots,a^k_d).\]
\end{definition}

Let $A$ and $B$ be sets and $f: B \to A$. If $k \in \N$ and $s \in B^k$, then by $f(s)$ we mean the tuple $(f(s_1), \dots, f(s_k)) \in A^k.$
We equip the space $A^B$ of functions from $B$ to $A$ with the topology of pointwise convergence, where $A$ is taken to be discrete. 
In this topology, a~basis of open sets is given by
\[{\mathscr S}_{s,t} := \{f \in A^B \mid f(s)=t\}\]
for $s \in B^k$ and $t \in A^k$ for some $k \in {\mathbb N}$.
For any topological space $T$, we denote by $\mathcal{B}(T)$
the Borel $\sigma$-algebra on $T$, i.e., the smallest subset of the powerset ${\mathcal P}(T)$ which contains all open sets and is closed under countable intersection and complement. We write $[0,1]$ for the set $\{x \in {\mathbb R} \mid 0 \leq x \leq 1\}$.

\begin{definition}[fractional map]
Let $A$ and $B$ be sets. A~\emph{fractional map} from $B$ to $A$ is a~probability distribution 
$(A^B, \mathcal{B}(A^B),\omega \colon \mathcal{B}(A^B) \to [0,1]),$
that is, $\omega(A^B) = 1$ and $\omega$ is countably additive: if $S_1, S_2,\dots \in \mathcal{B}(A^B)$ are disjoint, then \[\omega\left(\bigcup_{i \in {\mathbb N}} S_i\right) = \sum_{i \in {\mathbb N}} \omega(S_i).\]
\end{definition}

We often use $\omega$ for both the entire fractional map and for the map $\omega \colon \mathcal{B}(A^B) \to [0,1]$.

The set $[0,1]$ carries the topology inherited from  the standard topology on ${\mathbb R}$. We also view ${\mathbb R} \cup \{\infty\}$ as a~topological space with a~basis of open sets given by
all open intervals
$(a,b)$ for $a,b \in {\mathbb R}$, $a<b$ and additionally all sets of the form $\{x \in {\mathbb R} \mid x > a\} \cup \{\infty\}$ (thus, $\mathbb{R} \cup \{ \infty \}$ is equipped with its order topology when ordered in the natural way).

A \emph{(real-valued) random variable} is a~\emph{measurable function} $X \colon T \to {\mathbb R} \cup \{\infty\}$, i.e., pre-images of elements of $\mathcal{B}({\mathbb R} \cup \{\infty\})$ under $X$ are in $\mathcal{B}(T)$. 
If $X$ is a~real-valued random variable, then the \emph{expected value of $X$ (with respect to a~probability distribution $\omega$)} is denoted by $E_\omega[X]$ and is defined 
via the Lebesgue integral 
\[ E_\omega[X] := \int_T X d \omega. \]

Let $A$ and $B$ be sets. In the rest of the paper, we will work exclusively on a~topological space $A^B$ and the special case where $B=A^\ell$ for some $\ell \in \N$ and $A^B$ is the set of $\ell$-ary operations on $A$, we denote this set by ${\mathscr O}_A^{(\ell)}$.

\begin{definition}[fractional homomorphism]\label{def:frac-hom}
Let $\bA$ and $\bB$ be valued $\tau$-structures with domains $A$ and $B$, respectively. A~\emph{fractional homomorphism} from
$\bB$ to $\bA$ is a~
fractional map $\omega$ from $B$ to $A$
 such that for every $R \in \tau$ of arity $k$ 
and every tuple $t \in B^k$ 
it holds for the random variable $X \colon A^B \rightarrow \mathbb{R}\cup\{\infty\}$ given  by
$f \mapsto R^{\bA}(f(t))$ that 
$E_\omega[X]$
exists and that 
$E_\omega[X]
\leq R^{\bB}(t).$
\end{definition}

We refer to \cite{Resilience-VCSPs} for a~detailed introduction to fractional homomorphisms in full generality. \checked{We point out that fractional maps compose and the composition operation generalizes the standard composition of operations. Moreover, a composition of fractional homomorphisms is a fractional homomorphism~\cite[Proposition 5.9]{Resilience-VCSPs-arxiv}.}
Observe that if $\bA$ is a~countable valued structure and $R$ is a valued relation of $\bA$, we have the following handy expression for $E_\omega[f \mapsto R(f(t))]$, which we sometimes use in the proofs (see~\cite{Resilience-VCSPs} for a detailed derivation):
\begin{equation}
\label{eq:expr-expect}
E_{\omega}[f \mapsto R(f(t))] = \sum_{s \in A^k} R(\checked{s}) \omega(
{\mathscr S}_{t,s}).
\end{equation}

All concrete fractional homomorphisms $\omega$ that appear in this paper are of a~very special form, namely, there is a~single $f \in A^B$ such that $\omega(\{f\}) =1$. In this case, we also write $f$ instead of $\omega$.
If $\omega$ is of this form, then for every $R \in \tau$ of arity $k$ and $t \in B^k$, the expected value in Definition~\ref{def:frac-hom} always exists and is equal to $R^\bA(f(t))$.
If additionally $\bA$ and $\bB$ are crisp structures, then we call $f$ a~\emph{homomorphism}. It is easy to see that there are \emph{valued} structures $\bA$ and $\bB$ with a~fractional homomorphism from $\bB$ to $\bA$, but no fractional homomorphism from $\bB$ to $\bA$ of the form $f \in A^B$ (see, e.g.~\cite[Example 2.6 and 3.15]{ThesisZaneta}). 
  
\begin{lemma}\label{lem:frac-hom}
Let $\bA$ and $\bB$ be 
valued $\tau$-structures on countable domains such that $\Aut(\bA)$ is oligomorphic. If there exists a~fractional homomorphism from $\bB$ to $\bA$, then there also exists a~homomorphism from $\Feas(\bB)$ to $\Feas(\bA)$.
In particular, if $\bA$ and $\bB$ are crisp, then there is a fractional homomorphism from $\bB$ to $\bA$ if and only if there is a homomorphism from $\bB$ to $\bA$. 
\end{lemma}

\begin{proof}
Suppose that there exists a fractional homomorphism $\omega$ from $\bB$ to $\bA$. 
Since $B$ is countable and $\Aut(\bA) \subseteq \Aut(\Feas(\bA))$ is oligomorphic, it suffices to show that every finite substructure $\bF$ of $\Feas(\bB)$ has a~homomorphism to $\Feas(\bA)$ (see, e.g.,~\cite[Lemma 4.1.7]{Book}). 
Let $b_1,\dots,b_n$ be the elements of $\bF$ and $b:=(b_1, \dots, b_n)$. 
By the (countable) additivity of probability distributions, there
exists $a \in A^n$ such that $\omega(\mathscr{S}_{b,a})>0$.
Let $f$ be the map that takes $b$ to $a$. 
Suppose that there exists $R \in \tau$ of arity $k$ and $t \in F^k$ such that  
$R^{\bA}(f(t)) = \infty$. 
Since $\omega(\mathscr{S}_{b,a})>0$ and $\mathscr{S}_{b,a} \subseteq {\mathscr{S}}_{t,f(t)}$, we have  $\omega({\mathscr{S}}_{t,f(t)})>0$, and thus $E_\omega[g \mapsto R^{\bA}(g(t))] = \infty$ by~\eqref{eq:expr-expect}. 
Then $R^{\bB}(t) = \infty$, because $\omega$ is a~fractional polymorphism.
Hence, for every $R \in \tau$ of arity $k$ and $t \in F^k$ we have 
$R^{\bA}(f(t)) < \infty$ 
whenever 
$R^{\bB}(t) < \infty$. Therefore,  $f$ is a~homomorphism from $\bF$ to $\Feas(\bA)$.

The final statement follows from the fact that every homomorphism is a fractional homomorphism and that $\bC = \Feas(\bC)$ for every crisp structure $\bC$. 
\end{proof}

We say that two valued $\tau$-structures  $\bA$ and $\bB$ are \emph{fractionally homomorphically equivalent} if there exists a fractional homomorphisms from $\bA$ to $\bB$ and from $\bB$ to $\bA$. Since fractional homomorphisms compose~\cite{Resilience-VCSPs}, fractional homomorphic equivalence is indeed an equivalence relation on valued structures of the same signature. 

\begin{definition}[pp-construction]
\label{def:pp-constr}
Let $\bA, \bB$ be valued structures. Then  $\bB$ has a~\emph{pp-construction} in $\bA$ if $\bB$ 
is fractionally homomorphically equivalent to a~structure $\bB'$ which is a~pp-power of $\bA$. 
\end{definition}

\begin{remark}
We remark that our definition of pp-constructability (Definition~\ref{def:pp-constr}) applied to relational structures with oligomorphic automorphism groups coincides with the definition from~\cite{wonderland}. Let $\bA$ and $\bB$ be relational structures, and suppose that $\bA$ pp-constructs $\bB$, that is, there is a pp-power $\bA'$ of $\bA$ which is fractionally homomorphically equivalent to $\bB$.
Since $\bA$ is crisp, $\bA'$ is essentially crisp and $\Feas(\bA')$ is a pp-power of $\bA$ with all relations primitively positively definable in $\bA$ when viewed over the domain $A$. By Lemma~\ref{lem:frac-hom}, $\Feas(\bA')$ is homomorphically equivalent to $\bB$. 
This shows the claim. 
\end{remark}
The relation of pp-constructability is transitive: if $\bA$, $\bB$, and $\bC$ are valued structures such that $\bA$ pp-constructs $\bB$ and $\bB$ pp-constructs $\bC$, then $\bA$ pp-constructs $\bC$ \cite[Lemma 5.14]{Resilience-VCSPs}.

By $K_3$ we denote the complete graph on $3$ vertices.
The following is a~direct consequence of~\cite[Corollary 5.13]{Resilience-VCSPs} and \cite[Corollary 6.7.13]{Book}.
\begin{lemma}\label{lem:hard} 
Let $\bA$ be a~valued structure such that $\Aut(\bA)$ is oligomorphic. If $\bA$ pp-constructs $K_3$, then $\bA$ has a reduct $\bA'$ over a finite signature such that $\vcsp(\bA')$ is NP-hard.
\end{lemma} 

It is well-known that $K_3$ pp-constructs all finite relational structures (see, e.g, \cite[Corollary 6.4.4]{Book}). Hence, by the transitivity of pp-constructability, every valued structure that pp-constructs $K_3$ pp-constructs all finite relational structures. 

\subsection{Fractional polymorphisms}
Let $A$ be a~set and $R \subseteq A^k$. An operation $f \colon A^\ell \to A$ on the set $A$ \emph{preserves} $R$ if $f(t^1, \dots, t^\ell) \in R$ for every
$t^1,\dots, t^\ell \in R$.
If $\bA$ is a~relational structure and $f$ preserves all relations of $\bA$, then $f$ is called a~\emph{polymorphism} of $\bA$. The set of all polymorphisms of $\bA$ is denoted by $\Pol(\bA)$ and is closed under composition. We write $\Pol^{(\ell)}(\bA)$ for the set $\Pol(\bA) \cap \mathscr{O}_A^{(\ell)}$, $\ell \in \N$. Unary polymorphisms are called \emph{endomorphisms} and $\Pol^{(1)}(\bA)$ is also denoted by $\End(\bA)$. 

Let $\bA$ be a~relational structure and $\ell \in \N$. An operation $f \in \Pol^{(\ell)}(\bA)$ is called a~\emph{pseudo weak near unanimity (pwnu)} polymorphism if there exist $e_1, \dots, e_\ell \in \End(\bA)$ such that for every $x,y \in A$, $e_1f(y,x,\dots,x) = e_2 f(x,y,x,\dots,x) = \dots = e_\ell f(x,\dots,x,y).$
We now introduce  \emph{fractional polymorphisms} of valued structures, which generalize polymorphisms of relational structures.  Similarly to polymorphisms, fractional polymorphisms are an important tool for formulating tractability results and complexity classifications of VCSPs.
For valued structures with a~finite domain, our definition specialises to the established notion of a~fractional polymorphism which has been used 
to study the complexity of VCSPs for valued structures over finite domains (see, e.g.~\cite{ThapperZivny13}). Our definition is taken from \cite{Resilience-VCSPs} and allows arbitrary probability spaces in contrast to \cite{ViolaThesis,SchneiderViola,ViolaZivny}.

\begin{definition}[fractional operation]
Let $\ell \in {\mathbb N}$. 
A \emph{fractional operation on $A$ of arity $\ell$} is
a fractional map from $A^\ell$ to $A$. 
The set of all fractional operations on $A$ of arity $\ell$ is denoted by ${\mathscr F}^{(\ell)}_A$.
\end{definition}

\begin{definition}\label{def:pres}
A fractional operation $\omega \in {\mathscr F}_A^{(\ell)}$ \emph{improves}
a valued relation $R \in {\mathscr R}^{(k)}_A$ if for all $t^1,\dots,t^{\ell} \in A^k$ 
\begin{align}
E := E_\omega[f \mapsto R(f(t^1,\dots,t^{\ell}))]
\text{ exists, and } \quad 
E \leq
\frac{1}{\ell} \sum_{j = 1}^{\ell} R(t^j). \label{eq:fpol}
\end{align}
\end{definition}

Note that~\eqref{eq:fpol} has the interpretation that 
the expected value of $f \mapsto R(f(t^1,\dots,t^\ell))$ is at most the average of the values $R(t^1),\dots$, $R(t^\ell)$. 

\begin{definition}[fractional polymorphism]
A~fractional operation which improves every valued relation in a~valued structure $\bA$ is called a~\emph{fractional polymorphism 
of $\bA$}; the set of all fractional polymorphisms of $\bA$ is denoted by $\fPol(\bA)$.
\end{definition}

\begin{remark}\label{rem:frac-pol-frac-hom}
A fractional polymorphism of arity $\ell$ of a~valued $\tau$-structure $\bA$ might also be viewed as a~fractional homomorphism from a~specific $\ell$-th pp-power of $\bA$, which we denote by $\bA^{\ell}$, to $\bA$: 
the domain of $\bA^{\ell}$ is $A^{\ell}$,
and for every $R \in \tau$ of arity $k$ we have
\[R^{\bA^{\ell}}((a^1_1,\dots,a^1_{\ell}),\dots,(a^k_1,\dots,a^k_{\ell})) 
:= \frac{1}{\ell} \sum_{i=1}^\ell R^{\bA}(a^1_i,\dots,a^{k}_i).\]
\end{remark}

\begin{example}\label{expl:id}
Let $A$ be a~set and $\pi^\ell_i \in {\mathscr O}^{(\ell)}_A$ be the $i$-th projection of arity $\ell$, which is given by $\pi^\ell_i(x_1,\dots,x_\ell) = x_i$.
The fractional operation $\Id_\ell$ of arity $\ell$ such that $\Id_\ell(\pi^{\ell}_{i}) = \frac{1}{\ell}$ 
for every $i \in \{1,\dots,\ell\}$ is a~fractional polymorphism of every valued structure with domain $A$. 
\end{example}

As mentioned above, all concrete fractional polymorphisms $\omega$ that we need in this article are such that there exists an operation $f \in \mathscr{O}_A^{(\ell)}$ such that $\omega(\{f\})=1$.

\begin{lemma}[Lemma 6.8 in~\cite{Resilience-VCSPs}]\label{lem:easy-Imp-fPol}
Let $\bA$ be a~valued $\tau$-structure $\bA$ 
over a~countable domain $A$
and let $R$ be a valued relation over $A$. If 
$R \in \langle \bA \rangle$,
then $R$ is improved by all fractional polymorphisms of $\bA$.
\end{lemma}

\begin{remark}\label{rem:oligo}
    If $\bA$ is a valued structure such that $\Aut(\bA)$ is oligomorphic, then the converse of Lemma~\ref{lem:easy-Imp-fPol} might be true as well; this is for instance the case if $\bA$ is a \checked{relational} structure and $R$ a \checked{crisp} relation~\cite{BodirskyNesetrilJLC}.
\end{remark}

\begin{remark} \label{rem:ess-crisp-fPol}
Observe that $\Pol(\bA) \subseteq \fPol(\bA)$ for every crisp structure $\bA$. 
More concretely, for every $\ell \in \N$ we have that $\fPol^{(\ell)}(\bA)$ consists of precisely the fractional operations $\omega$ of arity $\ell$ such that $\omega(\Pol^{(\ell)}(\bA))=1$.
Also recall that, for essentially crisp valued structures $\bA$, $\langle \bA \rangle = \langle \Feas(\bA) \rangle$. 
Therefore,  Lemma~\ref{lem:easy-Imp-fPol} implies that $\fPol(\bA)=\fPol(\Feas(\bA))$.
\end{remark}

\subsection{Temporal valued structures}
Recall that a temporal valued \checked{structure} is a a~valued structure $\bA$ such that $\Aut(\Q; <) \subseteq \Aut(\bA)$. 
Every $\tau$-expression $\phi$ (in particular, every valued relation of $\bA$) attains only finitely many values in $\bA$: if $\phi = \phi(x_1, \dots, x_k)$, we may have at most one value for every orbit of the action of $\Aut(\bA)$ on $\Q^k$. Since the group $\Aut(\Q;<)$ has only finitely many orbits of $k$-tuples for every $k$, so does $\Aut(\bA)$. In particular, if $k=2$ and $a,b \in \Q$, $a<b$, then the values $\phi(a,a)$, $\phi(a,b)$, and $\phi(b,a)$ do not depend on the choice of $a$ and $b$. If $\Aut(\bA)=\Aut(\Q;=)$, then we additionally have $\phi(a,b)=\phi(b,a)$. 
As a~consequence, for every $\tau$-expression $\phi(x_1, \dots, x_k, y_1, \dots, y_\ell)$ and $b \in A^\ell$ there exists $a^* \in A^k$ such that \[\inf_{a \in A^k} \phi(a,b) = \min_{a \in A^k} \phi(a, b) =\phi(a^*, b).\]

The structure $(\Q;<)$ is known to be finitely bounded and homogeneous, which yields the following theorem as 
a~special case of \cite[Theorem 3.4]{Resilience-VCSPs}. 
\begin{theorem}\label{thm:NP}
Let $\bA$ be a~valued structure over a~finite signature such that $\Aut(\Q;<) \subseteq \Aut(\bA)$. Then $\VCSP(\bA)$ is in NP. 
\end{theorem}

To increase readability of some of the proofs, we
we will often use the following notation.
\begin{definition}[$E_t, N_t, O_t$]\label{def:E-N-O}
If $t \in \Q^k$ for some $k \in \N$, we define
\begin{align*}
E_t &:=\{(p,q) \in [k]^2 \mid t_p=t_q\},  \\
N_t &:=\{(p,q) \in [k]^2 \mid t_p\neq t_q\}, \text{ and} \\
O_t &:=\{(p,q) \in [k]^2 \mid t_p<t_q\},
\end{align*}
where $<$ is the natural order over $\Q$.
\end{definition}

As we already alluded to, we will repeatedly use the fact that in temporal VCSPs, any valued relation $R$ is such that $R(t)$ only depends on the \emph{order type} of the tuple $t$.
\begin{observation}\label{obs:order-type}
Let $\bA$ be valued structure with finite signature $\tau$ such that $\Aut(\Q;<) \subseteq \Aut(\bA)$ and let $R \in \tau$ be any valued relation of arity $k$.
Then for every $t, t' \in \Q^k$ such that $E_t = E_{t'}$ and $O_t = O_{t'}$ it holds that $R(t)=R(t')$.
\end{observation}

\section{General facts}
\label{sect:gen-facts}
In this section we \checked{prove}
some relatively easy, but very general and useful facts for analysing the computational complexity of VCSPs. 

\begin{lemma}\label{lem:constant-pol}
Let $\bA$ be a~valued structure with domain $A$ and finite relational signature $\tau$ such that there exists a~unary constant operation $c \in \fPol(\bA)$. 
Then $\vcsp(\bA)$ is in P. 
\end{lemma}
\begin{proof}
Suppose that there exists $b \in A$ such that the unary operation $c$ defined by $c(a) = b$ for all $a \in A$ is a~fractional polymorphism of $\bA$.
Then for every $R \in \tau$ of arity $k$ and $t \in A^k$, we have $R(b, \dots, b)=R(c(t)) \leq R(t)$.
Let $\phi(x_1,\dots,x_n) = \sum_i \phi_i$ be an instance of $\vcsp(\bA)$, where each $\phi_i$ is an atomic $\tau$-expression. 
Then the minimum 
of $\phi$ equals $\sum_i \phi_i(b,\dots,b)$ and hence
$\vcsp(\checked{\bA})$ is in P. 
\end{proof} 

The following proposition relates pp-constructability in a~valued structure $\bA$ with pp-constructability in the relational structure $(\Q; \langle \bA \rangle_0^\infty)$.

\begin{proposition}\label{prop:pp-constr-rel}
Let $\bA$ be a~valued structure and let $\bB$ be a~relational $\tau$-structure on countable domains $A$ and $B$, respectively. Then $\bA$ pp-constructs $\bB$ if and only if $(A; \langle \bA \rangle_0^\infty)$ pp-constructs $\bB$.
\end{proposition}

\begin{proof}
Clearly, if $(A; \langle \bA \rangle_0^\infty)$ pp-constructs $\bB$, then $\bA$ pp-constructs $\bB$. Suppose that $\bA$ pp-constructs $\bB$. Then there exists $d\in \N$ and a~pp-power $\bC$ on the domain $C=A^d$ of $\bA$ which is fractionally homomorphically equivalent to $\bB$. We claim that $\Feas(\bC)$ is fractionally homomorphically equivalent to $\bB$ as well, witnessed by the same fractional homomorphisms.

Let $\omega_1$ be a~fractional homomorphism from $\bC$ to $\bB$ and $\omega_2$ be a~fractional homomorphism from $\bB$ to $\bC$. Let $R \in \tau$ be of arity $k$ and $t \in C^k$. By the definition of a~fractional homomorphism,
\[E_{\omega_1}[f \mapsto R^{\bB}(f(t))] \leq R^{\bC}(t).\]
We claim that 
\begin{equation}
E_{\omega_1}[f \mapsto R^{\bB}(f(t))] \leq \Feas(R^{\bC})(t). \label{eq:pp-constr}
\end{equation}
This is clear if $\Feas(R^{\bC})(t)=\infty$. Otherwise, $\Feas(R^{\bC})(t)=0$, and therefore $R^{\bC}(t)$ is finite. Hence, 
\[E_{\omega_1}[f \mapsto R^{\bB}(f(t))] = \sum_{s \in B^k} \omega_1(\mathscr{S}_{t,s}) R^{\bB}(s)\]
is finite. 
Since $R^{\bB}$ attains only values $0$ and $\infty$, it follows that $E_{\omega_1}[f \mapsto R^{\bB}(f(t))]=0$. Therefore \eqref{eq:pp-constr} holds. Since $R$ and $t$ were arbitrary, it follows that $\omega_1$ is a~fractional homomorphism from $\Feas(\bC)$ to $\bB$. The proof that $\omega_2$ is a~fractional homomorphism from $\bB$ to $\Feas(\bC)$ is similar.

Note that $\Feas(\bC)$ is a~pp-power of $(A; \langle \bA \rangle_0^\infty)$: every relation $R$ of $\bC$ of arity $k$ lies in $\langle \bA \rangle$ when viewed as a~relation of arity $dk$ and therefore $\Feas(R) \in \langle \bA \rangle_0^\infty$. 
Hence, $(A; \langle \bA \rangle_0^\infty)$ pp-constructs $\bB$ as we wanted to prove.
\end{proof}

Let $A$ be a~set, $k\in \N$, and $i \in [k]$.
It is easy to see that the $k$-ary $i$-th projection $\pi_i^k \in \Pol(\bA)$ for every $k \in \N$, $i \in [k]$, and every relational structure $\bA$ on the domain $A$. 

\begin{lemma} \label{lem:ess-crisp}
Let $\bA$ be a~valued structure. Then $\fPol(\bA)$ contains $\pi^2_1$ if and only if 
$\bA$ is essentially crisp. 
\end{lemma} 

\begin{proof}
Suppose that $\bA$ contains a~valued relation $R$ of arity $k$ which takes two finite values $a$ and $b$ with $a<b$. 
Let $s \in A^k$ be such that $R(s) = a$ and $t \in A^k$ be such that $R(t) = b$. Then $R(\pi^2_1(t,s)) > \frac{R(s)+R(t)}{2}$, and hence $\pi^2_1 \notin   \fPol(\bA)$. 

Conversely, suppose that $\bA$ is essentially crisp. Let $R$ be a~valued relation of arity $k$ in $\bA$ and let 
$s,t \in A^k$. If $R(s) = \infty$ or $R(t) = \infty$
then $R(\pi^2_1(s,t)) \leq \frac{R(s)+R(t)}{2} = \infty$. Otherwise, $R(s) = R(t)$ and 
the inequality holds trivially. 
\end{proof}

\section{Equality VCSPs}
\label{sect:evcsps}
An \emph{equality structure} is a relational structure whose automorphism group is the group of all permutations of its domain~\cite{ecsps}; we define an \emph{equality valued structure} analogously. In this section, we prove that 
for every equality valued structure $\bA$, 
$\VCSP(\bA)$ is in P or NP-complete. This generalises the P versus NP-complete dichotomy for equality min-CSPs from \cite{OsipovW23}.

If the domain of $\bA$ is finite, then this is already known (see the discussion in the introduction). It is easy to see that classifying the general infinite case reduces to the countably infinite case. For notationally convenient use in the later sections, we work with the domain ${\mathbb Q}$, but we could have used any other countably infinite set instead.
We will need the following relation:
\begin{align} 
\label{eq:D}
\Dis & := \{(x,y,z) \in \Q^3 \mid (x=y \neq z) \vee (x \neq y = z) \}.
\end{align}
It is known that $(\Q;\Dis)$ pp-constructs $K_3$, see, e.g., Theorem 7.4.1 and Corollary 6.1.23 in \cite{Book}. 
Let $\const \colon {\mathbb Q} \to {\mathbb Q}$ be the constant zero operation, given by $\const(x) := 0$ for all $x \in {\mathbb Q}$. Let $\inj \colon \Q^2 \to \Q$ be injective.
A~tuple is called \emph{injective} if it has pairwise distinct entries.

\begin{theorem}[\hspace{1sp}\cite{Book, ecsps}]\label{thm:crisp-hard} 
If $\bA$ is a~relational structure such that $\Aut(\bA) = \Sym(\mathbb Q)$, then exactly one of the following cases applies. 
\begin{itemize}
    \item $\const \in \Pol(\bA)$ or $\inj \in \Pol(\bA)$.
    In this case, for every reduct $\bA'$ of $\bA$ with a~finite signature, $\CSP(\bA')$ is in P.
    \item The relation $\Dis$
    has a~primitive positive definition in $\bA$.
    In this case, $\bA$ pp-constructs $K_3$ and $\bA$ has a~reduct $\bA'$ with a~finite signature such that $\CSP(\bA')$ is NP-complete.
    \end{itemize}
\end{theorem}

The complexity classification of equality VCSPs is a consequence of several lemmas. The statements are inspired by analogous statements for equality CSPs that link the pp-definability of certain relations with the non-existence of certain polymorphisms~\cite{ecsps}. Our proofs, however, 
are quite remarkable, because we do not know yet whether a general preservation theorem holds that would link the pp-\checked{expressibility} of valued relations with the non-existence of \checked{certain} fractional polymorphisms; 
see the discussion in Remark~\ref{rem:oligo}.
\checked{Our first lemma shows a~useful consequence of $\const \notin \fPol(\bA)$.}

\begin{lemma}\label{lem:no-const-temp}
Let $\bA$ be a~valued structure such that $\Aut(\Q; <)\subseteq \Aut(\bA)$.
If $\const \notin \fPol(\bA)$, 
then 
$(\neq)_0^\infty \in \langle \bA \rangle$ or $(<)_0^\infty \in \langle \bA \rangle$. In particular, if $\Aut(\bA)=\Sym(\Q)$, then $(\neq)_0^\infty \in \langle \bA \rangle$. 
\end{lemma}

\begin{proof}
By assumption, there exists $R \in \tau$ of arity $k$ and $t \in A^k$ 
such that $m := R(t) < R(0,\dots,0)$. 
For $i \in \{1,\dots,k\}$, define $\psi_i(x_1, \dots, x_i)$ to be $R(x_1,\dots,x_i,x_i,\dots,x_i)$. 
Choose $t$ and $i$ such that $i$ is minimal with the property that $\psi_i(t_1,\dots,t_i) < R(0,\dots,0)$. 
Note that such an $i$ exists, because for $i = k$ we have 
$\psi_i(t_1,\dots,t_i) = R(t_1,\dots,t_k) < R(0,\dots,0)$. Moreover, $i>1$, since for every $a \in A$ there exists $\alpha \in \Aut(\bA)$ such that $\alpha(a)=0$ and hence $R(a, \ldots, a)=R(\alpha(a), \ldots, \alpha(a))=R(0, \ldots, 0)$. 
Also note that $t_{i-1} \neq t_i$, by the minimality of $i$. 

From all the such pairs $(t,i)$ that minimise $i$, choose a~pair 
$(t,i)$ where $\psi_i(t_1,\dots,t_i)$ is minimal. Such a~$t$ exists because $R$ attains only finitely many values. 
Define \[\psi(x_{i-1},x_i) := \min_{x_1,\dots,x_{i-2}} \psi_i(x_1,\dots, x_{i-2}, x_{i-1}, x_i).\]
Let $a\in A$ and note that, by~\cref{obs:order-type}, value $\psi(a,a)$ does not depend on the choice of~$a$. By our choice of~$i$, $\psi(a,a) > \psi(t_{i-1},t_i)$; otherwise, there are $a_1,\dots,a_{i-2} \in A$ such that 
$\psi_i(a_1,\dots,a_{i-2},a,a) \leq \psi_i(t_1, \dots, t_i)$, in contradiction to the choice of $(t,i)$ such that $i$ is minimal. 
We distinguish three cases (recall that $t_{i-1} \neq t_i$):
\begin{enumerate}
\item $\psi(t_{i-1}, t_i)=\psi(t_i,t_{i-1})$,
\item $\psi(t_{i-1}, t_i)<\psi(t_i,t_{i-1})$ and $t_{i-1}<t_i$,
\item $\psi(t_{i-1}, t_i)<\psi(t_i,t_{i-1})$ and $t_i < t_{i-1}$.
\end{enumerate}
Note that for $a,b \in A$ such that $a<b$, the values $\psi(a,b)$ and $\psi(b,a)$ do not depend on the choice of $a$, $b$.
In case 1, $\Opt(\psi)$ \checked{is equal to} $(\neq)_0^\infty$. In case 2, $\Opt(\psi)$ \checked{equals} $(<)_0^\infty$. Finally, in case 3 $\Opt(\psi)$ \checked{is equal to} $(>)_0^\infty$, \checked{from which we obtain} $(<)_0^\infty$ by exchanging the input variables.

The last statement follows from the fact that $\Sym(\Q)$ does not preserve $(<)_0^\infty$. 
\end{proof}

\checked{As a next step, we prove tractability for equality valued structures with an injective fractional polymorphism. The high level idea is that to find a solution to an instance, we collapse variables whose equality is forced and find an injective solution for the remaining variables.}

\begin{lemma}\label{lem:r}
Let $\bA$ be a~valued structure such that $\Aut(\bA) = \Sym(\Q)$.
Suppose that $\fPol(\bA)$ contains $\inj$. Then $\vcsp(\bA)$ is in P. 
\end{lemma}

\begin{proof}
Let $(\phi,u)$ be an instance of $\vcsp(\bA)$ with variable set $V$.
We first check whether $\phi$ contains summands with at most one variable that evaluate to $\infty$ for some (equivalently, for all) assignment; in this case, the minimum of $\phi$ is above every rational threshold and the algorithm rejects. 
Otherwise, we propagate (crisp) forced equalities:
if $\phi$ contains a~summand $R(x_1,\dots,x_k)$ 
and for all $f \colon V \to A$ we have that 
if $R(f(x_1),\dots,f(x_k))$ is finite, then $f(x_i) = f(x_j)$ for some $i<j$, then we say that $x_i = x_j$ is \emph{forced}.
In this case, we replace all occurrences of $x_j$ in $\phi$ by $x_i$ and repeat this process (including the check for unary summands that evaluate to $\infty$); clearly, this procedure must terminate after finitely many steps. 
Let $V'$ be the resulting set of variables, and let $\phi'$ be the resulting instance of $\vcsp(\bA)$. 
Clearly, the minimum for $\phi'$ equals the minimum for $\phi$.
Fix any injective $g \colon V' \to \Q$; we claim that $g$ minimises $\phi'$.
To see this, let 
$f \colon V' \to \Q$ be any assignment and let 
$\psi(x) = \psi(x_1,\dots,x_k)$ be a~ summand 
of $\phi'$. 
We show that $\psi(g(x)) \leq \psi(f(x))$. The statement is trivially true if $k=1$ by the transitivity of $\Aut(\bA)$.
Assume therefore that $k \geq 2$.

We first prove that $\psi(g(x))$ is finite.
Let $u^1,\dots,u^n$ be an enumeration of representatives of all orbits of $k$-tuples 
such that $\psi(u^i) < \infty$ and note that $n\geq 1$, because otherwise the algorithm would have rejected the instance. If for some distinct $p,q \in \{1,\dots,k\}$ 
we have  $(u^i)_p = (u^i)_q$ for all $i \in \{1,\dots,n\}$, then 
the algorithm would have replaced all occurrences of $x_p$ by $x_q$ or vice versa.
So
for all distinct $p,q \in \{1,\dots,k\}$ 
there exists $i \in \{1,\dots,n\}$ such that 
$(u^i)_p \neq (u^i)_q$.
Therefore, 
since $\inj$ is injective, the tuple $\inj(u^1,\inj(u^2,\dots,\inj(u^{n-1},u^n)\dots))$ lies in the same orbit as $g(x)$. Since $\inj \in \fPol(\bA)$, we have
$\psi(g(x)) < \infty$.

Note that
\begin{itemize}
\item 
$2 \psi(\inj(g(x),f(x)) \leq \psi(g(x)) + \psi(f(x))$, because $\inj \in \fPol(\bA)$, and 
\item $\inj(g(x),f(x))$ lies in the same orbit of $\Aut(\bA)$ as $g(x)$, and thus
$$\psi(\inj(g(x),f(x)) = \psi(g(x)).$$
\end{itemize}
Combining, we obtain 
that 
$\psi(g(x)) \leq \psi(f(x))$. 
It follows that $g$ minimises $\phi'$.
Recall that the minimum of $\phi$ and $\phi'$ are equal. Therefore, the algorithm accepts if the evaluation of $\phi'$ under $g$ is at most $u$ and rejects otherwise.
Since checking whether a~summand forces an equality can be done in constant time, and there is a~linear number of variables, the propagation of forced equalities can be done in polynomial time.
It follows that $\vcsp(\bA)$ is in P. 
\end{proof} 

\checked{The last auxiliary statement shows that if $\inj \notin \fPol(\bA)$, then this has strong consequences for expressive power of $\bA$. The proof idea is to use the existence of tuples that certify that $\inj \notin \fPol(\bA)$ to show that we can express some valued relations.}

\begin{lemma}\label{lem:no-r}
Let $\bA$ be a~valued structure such that $\Aut(\bA) = \Sym(\Q)$.
Suppose that $\inj \notin \fPol(\bA)$.
Then $(=)^1_0 \in \langle (\bA,(\neq)_0^\infty) \rangle$ or $\Dis \in \langle (\bA,(\neq)_0^\infty) \rangle$. 
\end{lemma}

\begin{proof}
By assumption, there exists $R \in \tau$ with arity $k$ which is not improved by $\inj$, that is, there exist $s,t \in \Q^k$ such that
\begin{equation}\label{eq:f-viol}
R(s)+R(t) < 2 R(\inj(s,t)).
\end{equation}
Note that, in particular, $R(s),R(t) < \infty$.
Suppose first that $R(\inj(s,t))=\infty$. In this case, the inequality above implies that $\Feas(R)$ is not improved by $\inj$.
It follows that $\Pol(\mathbb{Q}; \Feas(R), (\neq)_0^\infty)$ contains neither $\const$ nor $\inj$. 
Hence, by Theorem~\ref{thm:crisp-hard}, $\Dis$ is primitively positively definable in $(\mathbb{Q}; \Feas(R), (\neq)_0^\infty)$ 
and thus $\Dis \in \langle (\bA,(\neq)_0^\infty \rangle$. We may therefore assume in the rest of the proof that $R(\inj(s,t))<\infty$.

Inequality \eqref{eq:f-viol} implies that $R(s) < R(\inj(s,t))$ or $R(t) < R(\inj(s,t))$. Without loss of generality, assume $R(t)<R(\inj(s,t))$. Since $R(\inj(s,t))$ is finite, this implies \[R(\inj(s,t)) + R(t) < 2 R(\inj(s,t)) = 2 R(\inj(\inj(s,t),t)),\]
where the last equality follows from the fact that $\inj(s,t)$ and $\inj(\inj(s,t),t)$ lie in the same orbit of $\Aut(\bA)=\Sym(\mathbb{Q})$. This is an inequality of the same form as \eqref{eq:f-viol} (with $\inj(s,t)$ in the role of $s$), which implies that we can assume without loss of generality that $s$ and $\inj(s,t)$ lie in the same orbit of $\Aut(\bA)$.  
Then \eqref{eq:f-viol} implies that 
$R(t) < R(\inj(s,t))=R(s)$. We show that in this case $(=)^1_0 \in \langle (\bA,(\neq)_0^\infty) \rangle$.

Out of all pairs $(s,t) \in (A^k)^2$ such that $s$ and $\inj(s,t)$ lie in the same orbit and $R(t)<R(s)$, we choose $(s,t)$ such that $t_p\neq t_q$ holds for as many pairs $(p,q)$ as possible. 
Note that $s$ and $t$ cannot lie in the same orbit, and by the injectivity of $\inj$ there exist $i,j \in \{1,\dots,k\}$ such that
$t_{i} = t_{j}$ and $s_i \neq s_j$. 
Note that since $s$ and $\inj(s,t)$ lie in the same orbit, we have $E_s \subseteq E_t$ 
and $N_t \subseteq N_s$ (recall~\cref{def:E-N-O}). 
For the sake of the notation, assume that $i=1$ and $j=2$.
Consider the \checked{pp-}expression  
 \[\phi(x_1,x_2) := \min_{x_3,\dots,x_k} R(x_1,\dots,x_k) + \sum_{(p,q) \in E_s} (=)_0^\infty(x_p,x_q) + \sum_{(p,q) \in N_t} (\neq)_0^\infty(x_p,x_q).\]
 Then $\phi(x,y)$ attains at most two values by the $2$-transitivity of $\Aut(\bA)$. For every $x \in \Q$, we have $\phi(x,x) =: m \leq R(t)$. Let $\ell := \phi(x,y)$ for some distinct $x,y \in \Q$; this value does not depend on the choice of $x$ and $y$ by the $2$-transitivity of $\Aut(\bA)$. Suppose for contradiction that $\ell \leq m$. Then there exists a~tuple $u \in \Q^k$ such that
 \begin{enumerate}[(i)]
     \item $R(u) \leq R(t)$,
     \item $u_i \neq u_j$,
     \item $E_s \subseteq E_u$, and
     \item $N_t \subseteq N_u$.
 \end{enumerate}
By (iii), $\inj(s,u)$ lies in the same orbit as $s$. By (i), we get that $R(u) \leq R(t) < R(s)$. By (ii) and (iv), $u$ satisfies $u_p \neq u_q$ for more pairs $(p,q)$ than $t$, which contradicts our choice of $t$. Therefore, $m < \ell$. It follows that $\phi(x_1, x_2)$ 
is equivalent to $(=)_m^{\ell}$ with $m < \ell$. Recall that $(=)_0^\infty \in \langle \bA \rangle$ by definition and therefore $(=)_m^\ell \in \langle \bA, (\neq)_0^\infty \rangle$.
Note that $\ell \leq R(s) < \infty$. 
Hence, shifting $\phi$ by $-m$ and scaling it by $1/(\ell - m)$ shows that
$(=)^1_0 \in \langle \bA, (\neq)_0^\infty \rangle$, as we wanted to prove.
\end{proof} 

In the proof of the classification theorem for equality VCSPs below we use the following fact.
\begin{lemma}[see, e.g.,~{\cite[Example 4.9]{Resilience-VCSPs}}]\label{lem:soft-eq}
Let $\bA$ be a~valued structure such that $\Aut(\bA) = \Sym(\Q)$. 
If $\langle\bA\rangle$ contains $(=)^1_0$ and $(\neq)_0^{\infty}$, then $\Dis \in \langle \bA \rangle$. In particular, $\bA$ pp-constructs $K_3$, and, if the signature of $\bA$ is finite,
$\vcsp(\bA)$ is NP-complete. 
\end{lemma}

\begin{theorem}\label{thm:eqvcsp}
Let $\bA$ be a~valued structure 
with a~countably infinite domain $\Q$ over a~finite relational signature such that $\Aut(\bA) = \Sym(\Q).$
Then exactly one of the following two cases applies. 
\begin{enumerate}
    \item $(\neq)_0^\infty \in \langle \bA \rangle$ 
    and $\Dis \in \langle \bA \rangle$. In this case, 
$\bA$ pp-constructs $K_3$, and $\vcsp(\bA)$ is NP-complete. 
\item $\const\in \fPol(\bA)$ or $\inj\in \fPol(\bA)$.
In both of these cases, $\vcsp(\bA)$ is in P. 
\end{enumerate}
\end{theorem}

\begin{proof}
If $\const\in \fPol(\bA)$, then  Lemma~\ref{lem:constant-pol} implies that $\vcsp(\checked{\bA})$ is in P. So we suppose that $\const \notin \fPol(\bA)$ in the following; then $(\neq)_0^{\infty} \in \langle \bA \rangle$ by Lemma~\ref{lem:no-const-temp}.
If $\inj \in \fPol(\bA)$, then Lemma~\ref{lem:r} implies that $\vcsp(\bA)$ is in P. Therefore we also suppose that 
$\inj \notin \fPol(\bA)$ in the following. By Lemma~\ref{lem:no-r}, we have $(=)_0^1\in \langle \bA \rangle$ or  $\Dis \in \langle \bA \rangle$. In fact, Lemma~\ref{lem:soft-eq} implies that $\Dis \in \langle \bA \rangle$ in both cases. Then $\bA$ pp-constructs $K_3$ and $\VCSP(\bA)$ is NP-complete by Theorem~\ref{thm:crisp-hard}.
Note that neither $\const$ nor $\inj$ improves $\Dis$.
Therefore, the two cases in the statement of the theorem are disjoint.
\end{proof}


\begin{remark}
The complexity classification of equality minCSPs from \cite{equalityminCSP}, which can be viewed as VCSPs of valued structures where each relation attains only values $0$ and $1$, can be obtained as a~special case of Theorem~\ref{thm:eqvcsp}. Suppose that $\bA$ is such a~valued structure. If $\const \in \fPol(\bA)$, then $\bA$ is constant (in the terminology of~\cite{equalityminCSP}) and $\vcsp(\bA)$ is in P. 
If $\inj \in \fPol(\bA)$, then it is immediate that $\bA$ is Horn (in the terminology of~\cite{equalityminCSP}) and 
even strictly negative: otherwise, by~\cite[Lemma 16]{equalityminCSP} we have $(=)_0^1 \in \langle \bA \rangle$. But this is in contradiction to the assumption that 
$\inj \in \fPol(\bA)$, since $\inj$ applied to a~pair of equal elements and pair of distinct elements yields a~pair of
distinct elements,
increasing the cost to $1$ compared to the average cost $1/2$ of the input tuples.
Otherwise, it follows from Theorem~\ref{thm:eqvcsp} that $\vcsp(\bA)$ is NP-hard.
\end{remark}

\section{Temporal VCSPs}\label{sect:temp}
In this section we generalise the classification result from equality VCSPs to temporal VCSPs, which is the main result of this paper.

\subsection{Preliminaries on temporal CSPs} 
We first define several important relations on $\Q$  that already played a role  in the classification of temporal CSPs~\cite{tcsps-journal}; in this paper, we treat these relations in a black-box fashion. 
\begin{definition}
\label{def:rels}
Let
\begin{align*}
\Betw &:= \{(x,y,z) \in \Q^3 \mid (x<y<z) \vee (z<y<x) \}, \\
\Cyc &:= \{(x,y,z) \in \Q^3 \mid (x<y<z) \vee (y<z<x) \vee (z<x<y)\}, \\
\Sep &:= \{(x_1,y_1,x_2,y_2) \in \Q^4 
  \begin{aligned}[t]
    \mid~&(x_1 < x_2 < y_1 < y_2) \vee  (x_1 < y_2 < y_1 < x_2) \\
    \vee~&(y_1 < x_2 < x_1 < y_2)
    \vee (y_1 < y_2 < x_1 < x_2) \\
    \vee~&(x_2 < x_1 < y_2 < y_1) \vee  (x_2 < y_1 < y_2 < x_1) \\
    \vee~&(y_2 < x_1 < x_2 < y_1) \vee (y_2 < y_1 < x_2 < x_1) \},
  \end{aligned} \\
T_3 &:= \{(x,y,z) \in \Q^3 \mid (x=y<z) \vee (x=z<y)\}.
\end{align*}
\end{definition}

\begin{theorem}[Theorem 20 in \cite{tcsps-journal}]\label{thm:def-ord}
Let $\bA$ be a~relational structure with a~finite signature such that $\Aut(\Q;<)\subseteq\Aut(\bA)$. Then it satisfies at least one of the following:
\begin{itemize}
\item $\bA$ primitively positively defines $\Betw$, $\Cyc$, or $\Sep$.
\item $\const \in \Pol(\bA)$.
\item $\Aut(\bA)=\Sym(\Q)$.
\item There is a~primitive positive definition of $<$ in $\bA$.
\end{itemize}
\end{theorem}

To build on the results on temporal CSPs, we need the following operations on $\Q$; they will be used in a black-box fashion as well. 
By $\min$ and $\max$ we refer to the binary minimum and maximum operation on the set $\Q$, respectively. 

\begin{definition}
Let $e_{<0},e_{>0}$ be endomorphisms of $(\mathbb{Q};<)$ satisfying $e_{<0}(x)<0$ and $e_{>0}(x)>0$ for every $x\in \mathbb{Q}$.
We denote by $\pi\pi$  the binary operation on $\mathbb{Q}$ defined by 
\[\pi\pi(x,y) = \begin{cases}
      e_{<0} (x) &\text{if } x \leq  0, \\
      e_{>0}(y) &\text{if } x > 0.
\end{cases}
\]
\end{definition}

\begin{definition}  \label{def:lex_ll} 
Let $e_{<0},e_{>0}$ be endomorphisms of $(\mathbb{Q};<)$ satisfying $e_{<0}(x)<0$ and $e_{>0}(x)>0$ for every $x\in \mathbb{Q}$.
The operation $\lex$ is any binary operation on $ \mathbb{Q}$ satisfying $\lex(x,y) <\lex(x',y')$ iff $x<x'$, or $x=x'$ and $y<y'$ for all $x,x',y,y'\in \mathbb{Q}$.
We denote by $\lele$ the operation on $\mathbb{Q}$ defined by 
\[\lele(x,y) = \begin{cases}
      \lex(e_{<0}(x),e_{<0}(y)) &\text{if } x \leq  0, \\
      \lex(e_{>0}(y),e_{>0}(x)) &\text{if } x > 0.
\end{cases}
\]
\end{definition}

\begin{definition}\label{def:mi}
Let $e_<$, $e_=$ and $e_>$ be endomorphisms of $(\Q;<)$ satisfying
for all $x, \varepsilon \in \Q$, $\varepsilon > 0$,
$e_= (x) < e_{>}(x) < e_{<}(x) < e_{=}(x+\varepsilon).$
We denote by $\mi$  the binary operation on $\Q$ defined by
\[\mi(x,y) = \begin{cases}
      e_{<}(x) &\text{if } x < y, \\
      e_=(x) &\text{if } x = y, \\
      e_>(y) &\text{if } x > y.
\end{cases}
\]
\end{definition}

\begin{definition}\label{def:mx}
Let $e_=$ and $e_{\neq}$ be any endomorphisms
of $(\Q;<)$ satisfying
for all $x, \varepsilon \in \Q$, $\varepsilon > 0$, 
$e_{\neq}(x) < e_= (x) < e_{\neq}(x+\varepsilon).$
We denote by $\mx$  the binary operation on $\Q$ defined by
\[\mx(x,y) = \begin{cases}
      e_{\neq}(\min(x,y)) &\text{if } x \neq y, \\
      e_=(x) &\text{if } x = y.
\end{cases}
\]
\end{definition}

The construction of endomorphisms that appear in Definitions~\ref{def:mi} and~\ref{def:mx} can be found for example in \cite[Section 12.5]{Book}. \checked{We now give a few technical results that will be useful for the proofs in the upcoming sections.} The following \checked{lemma} was observed and used in~\cite{tcsps-journal}.

\begin{lemma}
\label{lem:lex}
If $\bA$ is a~relational structure 
such that $\Aut(\Q;<)\subseteq\Aut(\bA)$ 
and $\bA$ is preserved by a~binary injective operation $f$, then it is also preserved by the operation defined by one of $\lex(x,y)$, $\lex(-x,y)$, $\lex(x,-y)$, or $\lex(-x,-y)$. In particular, 
if $f$ preserves $\leq$ (for example, $\lele$), then $\bA$ is preserved by $\lex$. 
\end{lemma}

\begin{proposition}[Proposition 25, 27, and 29 in \cite{tcsps-journal}]\label{prop:pp}
Let $\bA$ be a~relational structure such that $\Pol(\bA)$ contains $\min$, $\mi$, or $\mx$.
Then $\Pol(\bA)$ contains $\pi\pi$.
\end{proposition}


\begin{theorem}[{\cite[Theorem 5.1]{BodirskyGreinerRydval}}]\label{thm:BGR}
Let $\bA$ be a~relational structure
with $\Aut(\bA)=\Aut(\Q;<)$ and suppose that $\bA$ contains the relation $<$.
If $\pi\pi \in \Pol(\bA)$ and $\lele \notin \Pol(\bA)$, 
then the relation
\[R^{\mix} = \{(x,y,z) \in \Q^3 \mid (x=y) \vee (z<x \wedge z<y) \}\]
has a~primitive positive definition in $\bA$.
\end{theorem}

\begin{definition}
The \emph{dual} of an operation $g \colon {\mathbb Q}^k \to {\mathbb Q}$ is the operation $$g^* \colon (x_1,\dots,x_k) \mapsto -g(-x_1,\dots,-x_k).$$
The \emph{dual} of a~relation $R \subseteq \Q^\ell$ is the relation
$$-R=\{(-a_1, \dots, -a_\ell) \mid (a_1, \dots, a_\ell) \in R \}.$$
\end{definition}
Note that $\min^*=\max$ and the relation $-(>)$ is equal to $<$. 
Statements about operations and relations on $\Q$ can be naturally dualised and we may refer to the dual version of a~statement.
%
%

By combining results from~\cite{tcsps-journal}
(Theorem~50, Corollary~51, Corollary~52, and the accompanying remarks), 
we obtain the following; also see Theorem 12.10.1 in~\cite{Book}. 
\begin{theorem} \label{thm:tcsp}
Let $\bA$ be a~relational structure with $\Aut(\Q; <)\subseteq \Aut(\bA)$. Then exactly one of the following is true.
\begin{enumerate}
    \item At least one of the operations $\const$, $\min$, $\mx$, $\mi$, $\lele$, or one of their duals lies in $\Pol(\bA)$. In this case, for every reduct $\bA'$ of $\bA$ with a~finite signature, $\CSP(\bA')$ is P. 
    \item $\bA$ primitively positively defines one of the relations $\Betw$, $\Cyc$, $\Sep$, $T_3$, $-T_3$, or $\Dis$. In this case, $\bA$ has a~reduct $\bA'$ with a~finite signature such that $\CSP(\bA')$ is NP-complete.
\end{enumerate}
Moreover, if $\bA$ has a finite signature, then it is decidable whether item 1 or item 2 holds.
\end{theorem}

We also need an alternative version of the classification theorem above. 

\begingroup
\setlength{\emergencystretch}{1em} 
\begin{theorem}[{{\cite[Theorem 12.0.1]{Book}}; see also {\cite[Theorem 7.24]{RydvalDescr}}}]
\label{thm:tcsp-pwnu}
Let $\bA$ be a~relational structure such that $\Aut(\Q;<) \subseteq \Aut(\bA)$. Then exactly one of the following is true:
\begin{enumerate}
\item $\Pol(\bA)$ contains a~pwnu polymorphism. In this case, for every reduct $\bA'$ of $\bA$ with a finite signature, $\CSP(\bA')$ is in P.
\item $\bA$
pp-constructs $K_3$. In this case, there exists a reduct $\bA'$ of $\bA$ with a finite signature such that $\CSP(\bA')$ is NP-complete.
\end{enumerate}
\end{theorem}
\endgroup

The following proposition relates hardness of temporal CSPs to pp-con\-struct\-ing $K_3$; it is essentially proven in \cite[Theorem 12.0.1]{Book}. 

\begin{proposition} \label{prop:hard}
The relational structures $(\Q; \Betw)$, $(\Q; \Cyc)$, $(\Q; \Sep)$, $(\Q; T_3)$, and $(\Q; -T_3)$ all pp-construct $K_3$.
\end{proposition}

\begin{proof}
This is proven in the proof of \cite[Theorem 12.0.1]{Book}. 
In fact, the proof shows that each of these structures pp-interprets all finite structures. Since $K_3$ is finite and a~pp-interpretation is a~special case of a~pp-construction, the statement follows.
\end{proof}

\subsection{A polynomial-time tractability result}
\label{sect:tract}
In this section we present the only non-trivial classes of polynomial-time
tractable temporal VCSPs.
In Lemma~\ref{lem:lex-algo} below we present a~polynomial-time algorithm for VCSPs of valued structures $\bA$ improved by $\lele$ or $\lele^*$. 





\begin{lemma} \label{lem:lex-algo}
Let $\bA$ be a~valued structure over a finite signature such that $\Aut(\Q; <)\subseteq \Aut(\bA)$. Suppose that $\checked{\lele} \in \fPol(\bA)$ \checked{or $\lele^* \in \fPol(\bA)$.}
Then $\VCSP(\bA)$ is in P.
\end{lemma}

In Algorithm~\ref{algo:lex} we provide the polynomial-time algorithm that is described in 
the proof of Lemma~\ref{lem:lex-algo}.
\checked{The high level idea of the algorithm is similar to the algorithm in Lemma~\ref{lem:r}: we detect forced equalities and collapse the corresponding variables, then we find an injective solution to the modified instance.}

\begin{algorithm}[t!]
\SetAlgoLined
\caption{\label{algo:lex} An algorithm for $\VCSP (\bA)$  where $\bA$ \checked{is improved by $\lele$ or $\lele^*$}.}  
\KwIn{An instance $(\phi,u)$ of $\VCSP(\bA)$ with variables $V = \{v_1, \dots, v_N\}$}  
\KwOut{\emph{true} or \emph{false}}   
\tcp{$\CSP(\Feas(\bA))$ is solved by the algorithm from~\cite{tcsps-journal} based on $\lele$ or $\lele^*$}
\If{$\phi$ satisfiable over $\Feas(\bA)$}
{ $E := \{(v_i, v_j) \mid f(v_i) = f(v_j) \text{ for every solution } f: V \to \Q \text{ to $\phi$ in } \Feas(\bA)\}$}
\Else{\Return \emph{false}}
$\phi' := \phi$ where $v_j$ is replaced by $v_i$ for every $(v_i,v_j) \in E$ with $i<j$

$\phi' = \phi_1'+ \dots + \phi_n'$ where $\phi_j'$ are atomic

\For{$j \in [n]$}
{$\phi_j' = R(x_1, \dots, x_k)$ for some $R \in \tau$ \\
let $y_1^{j}, \dots y_{\ell_j}^j$ be an enumeration of all distinct variables in $\{x_1, \dots, x_k\}$ \\
$S_j(y_1^{j}, \dots y_{\ell_j}^j) := R(x_1, \dots x_k)$
}
$\bB := (\Q; \Opt(S_1), \dots, \Opt(S_N))$ 

\tcp{by the assumptions $\CSP(\bB)$ is in P}

$\psi := \Opt(S_1)(y_1^{1}, \dots y_{\ell_1}^1) \wedge \dots \wedge \Opt(S_n)(y_1^n, \dots y_{\ell_n}^n)$

\tcp{it is implied that $\psi$ is satisfiable over $\bB$}

$f:=$ solution to $\psi$ over $\bB$

\If{the cost of $\phi'$ evaluated under $f$ is $\leq u$}
{\Return \emph{true}}
\Else{\Return \emph{false}}
   
\end{algorithm} 

\begin{proof}[Lemma~\ref{lem:lex-algo}]
Let $R \in \langle \bA \rangle$ be of arity $k$. 
Since \checked{$\lele$ or $\lele^*$ lies in $\fPol(\bA)$,} $\lex \in \fPol(\bA)$ \checked{by Lemma~\ref{lem:lex}}, \checked{and, in particular, $\lex$} improves $R$ \checked{by Lemma~\ref{lem:easy-Imp-fPol}}. Therefore, for every injective tuple $s \in \Q^k$ and any $t \in \Q^k$, it holds that
\begin{equation*}
R(s)=R(\lex(s,t)) \leq 1/2 \cdot (R(s)+R(t)),
\label{eq:lex}
\end{equation*}
where the first equality follows from $s$ and $\lex(s,t)$ being in the same orbit of $\Aut(\bA)$.
Therefore, if $R(s) < \infty$, then $R(s) \leq R(t)$.  In particular, there is $m_R \in \Q$ such that for every injective tuple $s \in \Q^k$, we have $R(s)=m_R$ or $R(s)=\infty$. Note that if there is at least one injective tuple $s$ with $R(s)=m_R$, then $\Opt(R)$ is the crisp relation that consists of all the tuples $t$ such that $R(t) = m_R$.

Let $(\phi,u)$ be an instance of $\vcsp(\bA)$ with variable set $V = \{v_1, \dots, v_N\}$.
Note that if we interpret $\phi$ over $\Feas(\bA)$, we can view it as an instance of $\CSP(\Feas(\bA))$ where each summand $R(x_1, \dots, x_k)$ of $\phi$ is interpreted as $\Feas(R^\bA)(x_1, \dots, x_k)$.
By the assumption \checked{and Lemma~\ref{lem:easy-Imp-fPol}}, 
$\Feas(\bA)$ is preserved by $\lele$ or $\lele^*$. 
Hence, by Theorem~\ref{thm:tcsp}, $\CSP(\Feas(\bA))$ is solvable in polynomial time and we can use the polynomial-time algorithm from \cite{tcsps-journal} based on the operation $\lele$ or $\lele^*$ to solve $\CSP(\Feas(\bA))$.
If $\phi$, viewed as a~primitive positive formula, is not satisfiable over $\Feas(\bA)$, then the minimum of $\phi$ is above every rational threshold and $(\phi,u)$ is rejected. Otherwise, 
we may compute the set $E \subseteq V^2$ of all pairs $(x,y)$ such that $f(x)=f(y)$ in every solution of $f \colon V \to \Q$ of $\phi$
over $\Feas(\bA)$ (we may assume without loss of generality that $\Feas(\bA)$ contains the relation $(\neq)_0^\infty$; since $\Feas(\bA)$ is preserved by $\lex$ it suffices to test the unsatisfiability of $\phi \wedge x \neq y$ for each of these pairs).  
It follows from the definition of $\Feas$ that for every $g \colon V \to \Q$, if $\phi$ evaluates to a~finite value in $\bA$ under the assignment $g$, then $g(x)=g(y)$ for every $(x,y) \in E$. Moreover, for every $(x,y) \in V^2\setminus E$, there exists $g \colon V \to \Q$ such that $\phi$ evaluates to a~finite value under $g$ and $g(x) \neq g(y)$.

We create a~new $\tau$-expression $\phi'$ from $\phi$ by replacing each occurrence of $v_j$ by $v_i$ for every $(v_i, v_j) \in E$ such that $i<j$. Let $V'$ be the set of variables of $\phi'$.
By the discussion above, the minimum for $\phi'$ over $\bA$ equals the minimum for $\phi$.
Moreover, for every $(x,y) \in (V')^2$, there exists $g' \colon V' \to \Q$ such that $\phi'$ evaluates to a~finite value under $g'$ and $g'(x) \neq g'(y)$.
Let $\phi' := \phi'_1 + \dots + \phi'_n$ where for every $j \in [n]$ the summand $\phi'_j$ is an atomic $\tau$-expression. We execute the following procedure for each $j \in [n]$. Let $\phi'_j=R(x_1, \dots, x_k)$. Let $y_1^j, \dots, y_{\ell_j}^j$ be an enumeration of all distinct variables that appear in $\{x_1, \dots, x_k\}$ and let $S_j$ be a~valued relation of arity $\ell_j$
defined by $S_j(y_1^j, \dots, y_{\ell_j}^j)=R(x_1, \dots, x_k)$. Clearly, $S_j$ is a~minor of $R$. Note that the relation $S_j$ might be different for every summand, even if they contain the same relation symbol $R$, due to possibly different variable identifications. Observe that, by the properties of $\phi'$, there exists an injective tuple $s^j \in \Q^{\ell_j}$ such that $S_j(s^j)$ is finite.
Note that $S_j \in \langle \bA \rangle$, and let $m_j := m_{S_j}$.
By the discussion in the beginning of the proof, $S_j(s^j)=m_{j}$ and $\Opt(S_j) \in \langle \bA \rangle_0 ^\infty$ consists of all tuples that evaluate to $m_{j}$ in $S_j$. Since $S_j$ attains only finitely many values, we can identify $m_j$ in polynomial time for every $j$.

Let $\bB$ be the relational structure with domain $\mathbb Q$ and relations $\Opt(S_1)$, $\dots$, $\Opt(S_n)$. Let $\psi$ be the instance of $\CSP(\bB)$ obtained from $\phi'$ by replacing the summand $\phi'_j$ by $\Opt(S_j)(y_1^j, \dots, y_{\ell_j}^j)$ for all $j \in [n]$; all relations in $\psi$ are crisp and hence it can be seen as a~primitive positive formula. Note that the variable set of $\psi$ is equal to $V'$.
By assumption \checked{and Lemma~\ref{lem:easy-Imp-fPol},}
\checked{$\bB$} is preserved by  
$\lele$ or $\lele^*$.
Hence, $\CSP(\bB)$ is in P by Theorem~\ref{thm:tcsp}.
Therefore, the satisfiability of $\psi$ over $\bB$ can be tested in polynomial time.
We claim that if $\psi$ is unsatisfiable, then the minimum of $\phi$ is above every rational threshold and the algorithm rejects.

We prove the claim by contraposition. 
Suppose that the minimum of $\phi$ over $\bA$ is finite. Then the minimum of $\phi'$ over $\bA$ is finite and hence there exists $f' \colon V' \to \Q$ such that $\phi'$ evaluates to a~finite value under $f'$. From all $f'$ with this property, choose $f'$ with the property that $f'(x)\neq f'(y)$ holds for as many pairs $(x,y) \in (V')^2$ as possible.
We first show that $f'$ is in fact injective. Suppose that there are $v, w \in V'$ such that $f'(v)=f'(w)$. Let $g' \colon V' \to \Q$ be such that $\phi'$ evaluates to a~finite value under $g'$ and $g'(v) \neq g'(w)$; recall that such $g'$ must exist by the construction of $\phi'$. Consider the assignment $\lex(f',g') \colon V' \to \Q$ and note that $\lex(f',g')(x) \neq \lex(f',g')(y)$ holds for all pairs $(x,y)$ such that $f'(x) \neq f'(y)$ and also $\lex(f',g')(v) \neq \lex(f',g')(w)$. Moreover, $\phi'$ evaluates to a~finite value under $\lex(f',g')$: for every $j \in [n]$, if $\phi'_j$ is of the form $R(x_1, \dots, x_k)$, then, since $\lex \in \fPol(\bA)$, 
\[R(\lex(f',g')(x_1, \dots, x_k))\leq 1/2 \cdot (R(f'(x_1, \dots x_k))+R(g'(x_1, \dots, x_k))) < \infty.\]
This contradicts our choice of $f'$. Therefore, $f'$ is injective.

Note that for every $j \in [n]$ we have $S_j(f'(y_1^j), \dots f'(y_{\ell_j}^j))<\infty$, because $\phi_j'$ evaluates to a~finite value under $f'$. Since $(f'(y_1^j), \dots f'(y_{\ell_j}^j))$ is an injective tuple, this implies $S_j(f'(y_1^j), \dots f'(y_{\ell_j}^j)) = m_j$ and $(f'(y_1^j), \dots f'(y_{\ell_j}^j)) \in \Opt(S_j)$ for every $j \in [n]$. It follows that $f'$ is a~satisfying assignment to $\psi$. Therefore, we proved that whenever $\psi$ unsatisfiable, the algorithm correctly rejects, because there is no assignment to $\phi$ of finite cost.

Finally, suppose that there exists a~solution $h' \colon V' \to \Q$ to the instance $\psi$ of $\CSP(\bB)$.
Then, for every $j \in [n]$, $\phi'_j$ takes under $h'$ the value $S_j(h'(y_1^j), \dots, h'(y_{\ell_j}^j))$. By the definition of $\Opt$, $(h'(y_1^j), \dots, h'(y_{\ell_j}^j))$ minimizes $S_j$ and therefore $h'$ minimizes $\phi'_j$. It follows that $h'$ minimizes $\phi'$ and that the cost of $\phi'$ under $h'$ is equal to $m_1 + \dots + m_n$. Since the minimum of $\phi'$ is equal to the minimum of $\phi$, the algorithm accepts if 
$m_1 + \dots + m_n \leq u$
and rejects otherwise.
This completes the algorithm and its correctness proof. It follows that $\vcsp(\bA)$ is in P. 
\end{proof}

\subsection{Expressibility of hard valued relations} 
In this section we prove several lemmas that are needed to complete the proof of the complexity classification of temporal VCSPs.
%
We introduce the following notation.
For $\alpha,\beta,\gamma \in {\mathbb Q} \cup \{\infty\}$, 
define the binary valued relation
$R_{\alpha,\beta,\gamma}$
on ${\mathbb Q}$:
\begin{align*}
    R_{\alpha,\beta,\gamma}(x,y) := \begin{cases} 
    \alpha & x=y \\
    \beta & x<y \\
    \gamma & x > y 
    \end{cases}
\end{align*}
Note that $R_{0,1,1}$ is equal to $(=)_0^1$, $R_{1,0,0}$ is equal to $(\neq)_0^1$, $R_{1,0,1}$ is equal to $(<)_0^1$, and $R_{0,0,1}$ is equal to $(\leq)_0^1$.

\begin{lemma} \label{lem:soft-ord}
Let $\bA$ be a~valued structure such that $\Aut(\Q;<) \subseteq \Aut(\bA)$ and $\alpha > \frac{1}{3}$.
If $\langle \bA \rangle$ contains
$R_{\alpha,0,1}$, 
then 
$\Cyc \in \langle \bA \rangle$.
\end{lemma}

\begin{proof}
Note that 
\[\Cyc(x,y,z) = \Opt( R_{\alpha,0,1}(x,y)+R_{\alpha,0,1}(y,z) + R_{\alpha,0,1}(z,x)).\] Therefore, $\Cyc \in \langle \bA \rangle$.
\end{proof}

\begin{lemma}\label{lem:3-val-rels}
Let $\bA$ be a~valued structure such that $\Aut(\Q;<) \subseteq \Aut(\bA)$. Let $(<)_0^\infty \in \langle \bA \rangle$. Let $\alpha, \beta, \gamma \in \mathbb{Q} \cup \{\infty\}$ be such that 
\begin{itemize}
\item $\alpha < \min(\beta, \gamma) < \infty$, or
\item $\beta \neq \gamma$ and $\beta, \gamma < \infty$.
\end{itemize}
If $R_{\alpha, \beta, \gamma} \in \langle \bA \rangle$, then
$\Cyc \in \langle \bA \rangle$. 
\end{lemma}

\begin{proof}
Without loss of generality, we may assume that $\beta \leq \gamma$, because otherwise we may consider $R_{\alpha, \gamma, \beta}(x,y) = R_{\alpha, \beta, \gamma} (y,x)$.
In the first case of the statement we have that $\alpha < \min(\beta, \gamma) = \beta < \infty$. Then
\[(<)_0^1(x,y) = \frac{1}{\beta-\alpha} \min_{z} \left( R_{\alpha, \beta, \gamma}(z,x) + (<)_0^\infty(z,y) - \alpha \right).\]

Suppose now that we are not in the first case, i.e., $\alpha \geq \min(\beta, \gamma) = \beta$, and additionally
$\beta < \gamma < \infty$ 
as in the second case of the statement.
Then for every $x,y \in \Q$
\[(<)_0^1(x,y)= \frac{1}{\gamma - \beta} \min_{z} \left( R_{\alpha, \beta, \gamma}(x,z) + (<)_0^\infty(z,y) - \beta \right).\]
Therefore, in both cases, $(<)_0^1 \in \langle \bA \rangle$. Since $(<)_0^1$ is equal to $R_{1,0,1}$, the statement follows from Lemma~\ref{lem:soft-ord}.
\end{proof}

\checked{The following lemma shows that if $\bA$ is not an equality valued structure and has a~tractable VCSP, then $\const \notin \fPol(\bA)$ implies that $(<)_0^\infty$ is pp-expressible in $\bA$. The proof is based on isolating an expressible relation not preserved by $\Sym(\Q)$ and Theorem~\ref{thm:def-ord}.}

\begin{lemma} \label{lem:expr-ord}
Let $\bA$ be a~valued structure such that $\Aut(\Q; <)\subseteq \Aut(\bA) \neq \Sym(\Q)$.
If $\const \notin \fPol(\bA)$, 
then $\langle\bA\rangle$ contains  
$\Betw$, $\Cyc$, $\Sep$, 
or $(<)_0^\infty$. 
\end{lemma}

\begin{proof}
By Lemma~\ref{lem:no-const-temp}, $\langle \bA \rangle$ contains $(<)_0^\infty$ or $(\neq)_0^\infty$. If $(<)_0^\infty \in \langle \bA \rangle$, then we are done. Assume therefore $(\neq)_0^\infty \in \langle \bA \rangle$. Let $R$ be a~valued relation of $\bA$ of arity $k$ such that there exists an orbit $O$ of the action of $\Sym(\Q)$ on $\Q^k$ and $s,t \in O$ with $R(s)<R(t)$. Let $s \in O$ be such that $R(s)$ is minimal. Note that $O$ is not the orbit of constant tuples, because $\Aut(\Q;<)$ is transitive.

Consider the crisp relation $S \in \langle \bA \rangle_0^\infty$ defined by the following expression with the free variables $x_1,\dots,x_k$. 
\[
\Opt \left( R(x_1, \dots, x_k) + \sum_{(p,q) \in E_s} (=)_0^\infty(x_p,x_q) + \sum_{(p,q) \in N_s} (\neq)_0^\infty(x_p,x_q) \right)\]
Clearly, $s \in S$.
Note that a~tuple $u \in \Q^k$ lies in $O$ if and only if
$E_s \subseteq E_u$ and $N_s \subseteq N_u$.
In particular, $S \subseteq O$.  Since $R(s) < R(t)$, we have $t \notin S$. It follows that $S$ is not preserved by $\Sym(\Q)$. Moreover, $S$ is not preserved by $\const$,  because $O$ is not the orbit of constant tuples. By Theorem~\ref{thm:def-ord}, the relational structure $(\Q; S)$
admits a~primitive positive definition of $\Betw$, $\Cyc$ or $\Sep$, 
or a~primitive positive  definition of $<$. Since $S \in \langle \bA \rangle$, the statement of the lemma follows.
\end{proof}

\checked{As a next step, we use the assumption that $\bA$ is not essentially crisp and the relation $(<)_0^\infty$ to express some useful valued relations.}

\begin{lemma}\label{lem:soft-neq}
Let $\bA$ be a~valued $\tau$-structure such that $\Aut(\bA) = \Aut(\Q; <)$ and $(<)_0^\infty \in \langle\bA\rangle$. Suppose that $\bA$ is not essentially crisp. Then one of the following holds:
\begin{itemize}
    \item $\Cyc \in \langle \bA \rangle$,
    \item $(\neq)_0^1 \in \langle \bA \rangle$, or
    \item $R_{1,0,\infty} \in \langle \bA \rangle$.
\end{itemize}
\end{lemma}

\begin{proof}
Let $R$ be a~valued relation of $\bA$ of arity $k$ that attains at least two finite values. Let $m,\ell \in {\mathbb Q}$, with $m < \ell$, be the two smallest finite values attained by $R$. Let $t \in \Q^k$ be such that $R(t)=\ell$. 
Choose $s \in \Opt(R)$ so that $|(E_s \cap E_t) \cup (O_s \cap O_t)|$ is maximal (recall~\cref{def:E-N-O}). 
Clearly, $R(s) = m$.

Let $\sim \, \subseteq (\Q^2)^2$ be the equivalence relation with the classes $=$, $<$, and $>$.
Since $R(s) \neq R(t)$, there exist distinct $i,j$ such that $(s_i,s_j) \not \sim (t_i, t_j)$. For the sake of notation, assume that $(i,j)=(1,2)$. Let $\phi$ be the following expression with the free variables $x_1,x_2$.
\begin{align*} \min_{x_3, \dots, x_k} 
\left(
R(x_1, \dots, x_k) + \sum_{(p,q) \in E_s \cap E_t} (=)_0^\infty(x_p,x_q) + \sum_{(p,q) \in O_s \cap O_t} (<)_0^\infty(x_p,x_q) \right)
\end{align*}
Observe that $\phi(x,y) \geq m$ for all $x, y \in \Q$ and hence whenever $(x,y) \sim (s_1, s_2)$ we have $\phi(x,y) = m$. Let $(x,y) \sim (t_1, t_2)$. Then $\phi(x,y) \leq \ell$. By the choice of $s$, there is no $s' \in \Opt(R)$ that satisfies $(s'_1, s'_2) \sim (t_1, t_2)$, $(E_s \cap E_t) \subseteq E_{s'}$ and $(O_s \cap O_t) \subseteq O_{s'}$. Therefore, $\phi(x,y) >m$. It follows that $\phi(x,y) = \ell$. 

Let 
\[S(x,y) := \frac{1}{\ell-m} (\phi(x,y)-m).\]
By the construction, $S \in \langle \bA \rangle$, $S(x,y)=0$ for $(x,y) \sim (s_1,s_2)$, and $S(x,y)=1$ for $(x,y) \sim (t_1, t_2)$. Note that $\Aut(\Q;<)$ has three orbits of pairs, two of which are represented by $(s_1, s_2)$ and $(t_1, t_2)$. Let $(u_1, u_2) \in \Q^2$ be a~representative of the third orbit and let $\alpha =S(u_1, u_2)$. It follows that $S$ is equal to one of the relations $R_{0,1,\alpha}$, $R_{0,\alpha,1}$, $R_{1,0,\alpha}$, $R_{1,\alpha,0}$, $R_{\alpha,0,1}$ or $R_{\alpha,1,0}$. By the choice of $m$ and $\ell$, we have that $\alpha=0$ or $\alpha \geq 1$. By Lemma~\ref{lem:3-val-rels}, this implies that 
$\Cyc \in \langle \bA \rangle$ 
unless $S = R_{1,0,0}$, $S=R_{1,0,\infty}$, or $S=R_{1,\infty,0}$. Since $R_{1,0,0}$ is equal to $(\neq)_0^1$ and $R_{1,0, \infty}(x,y) = R_{1,\infty, 0}(y,x)$, the statement follows.
\end{proof}

\checked{The following lemma relates expressibility of hard relations to not having $\lele$ and $\lele^*$ in $\Pol(A; \langle \bA \rangle_0^\infty)$. The essential ingredient of the proof is Theorem~\ref{thm:BGR}.}
\begin{lemma}\label{lem:lex-viol-crisp}
Let $\bA$ be a~valued structure such that $\Aut(\bA)=\Aut(\Q;<)$ and $(<)_0^\infty \in \langle \bA \rangle$. Suppose that $\checked{\lele, \lele^*} \notin \Pol(\Q; \langle \bA \rangle_0^\infty)$ and that $\bA$ is not  essentially crisp.
Then $\langle \bA \rangle$ contains one of the relations $\Betw$, $\Cyc$, $\Sep$, $T_3$, $-T_3$, or $\Dis$. 
\end{lemma}

\begin{proof}
Let $\bA' := (\Q; \langle \bA \rangle_0^\infty)$. Note that $\const \notin \Pol(\bA')$, because $\const$ does not preserve $(<)_0^\infty$. If 
$\langle \bA \rangle_0^\infty$ contains one of the relations $\Betw$, $\Cyc$, $\Sep$, $T_3$, $-T_3$, or $\Dis$, then we are done. Suppose that this is not the case.
Then $\Pol(\bA')$ contains \checked{$\min$, $\mx$, $\mi$, or one of their duals by Theorem~\ref{thm:tcsp} and the assumptions of the lemma.} Suppose first that $\Pol(\bA')$ contains $\min$, $\mx$, $\mi$.
By Proposition~\ref{prop:pp}, $\Pol(\bA')$ contains $\pi\pi$.
By Theorem~\ref{thm:BGR}, $\bA'$ primitively positively defines, equivalently, contains the relation $R^{\mix}$. 
By Lemma~\ref{lem:soft-neq}, we have that
$\langle \bA \rangle$ contains $\Cyc$, $(\neq)_0^1$, or $R_{1,0,\infty}$. If $\Cyc \in \langle \bA \rangle$, then we are done. Suppose therefore that $(\neq)_0^1 \in \langle \bA \rangle$ or $R_{1,0,\infty} \in \langle \bA \rangle$.
Note that for every $x,y \in \Q$, we have
\begin{align*}
(<)_0^1(x,y) &= \min_z \big(R^{\mix}(y,z,x) + (\neq)_0^1(y,z) \big) \\
& = \min_z \big(R^{\mix}(y,z,x) + R_{1,0,\infty}(y,z) \big).
\end{align*}

Indeed, if $x<y$, then by choosing $z>y$ we get $R^{\mix}(y,z,x) + (\neq)_0^1(y,z) = R^{\mix}(y,z,x) + R_{1,0,\infty}(y,z) = 0$, which is clearly the minimal value that can be obtained. If $x \geq y$, then by choosing $z=y$ we get $R^{\mix}(y,z,x) + (\neq)_0^1(y,z) = R^{\mix}(y,z,x) + R_{1,0,\infty}(y,z) = 1$, which is clearly the minimal value, 
because if $z \neq y$ we obtain $R^{\mix}(y,z,x) + (\neq)_0^1(y,z) = R^{\mix}(y,z,x) + R_{1,0,\infty}(y,z) = \infty$.

It follows that $(<)_0^1 \in \langle \bA \rangle$. 
Observe that $(<)_0^1$ equals $R_{1,0,1}$. Therefore,
$\Cyc \in \langle \bA \rangle$ by Lemma~\ref{lem:soft-ord}, as we wanted to prove.
If $\Pol(\bA')$ contains $\min^*$, $\mx^*$, \checked{ or} $\mi^*$,
we use the dual versions of Proposition~\ref{prop:pp} and Theorem~\ref{thm:BGR} to analogously prove that $(>)_0^1 \in \langle \bA \rangle$. 
Since $(<)_0^1(x,y) = (>)_0^1(y,x)$ 
we obtain $\Cyc \in \langle \bA \rangle$  by Lemma~\ref{lem:soft-ord}.
\end{proof}

\checked{The last auxiliary lemma needed for the classification of temporal VCSPs shows that, assuming $(<)_0^\infty \in \langle \bA \rangle$, if an injective operation $f$ preserves all crisp relations expressible in $\bA$, then $f$ is a fractional polymorphism of $\bA$ or  $\VCSP(\bA)$ is NP-hard.}

\begin{lemma}\label{lem:lex-viol-soft}
Let $\bA$ be a~valued structure with $\Aut(\bA)=\Aut(\Q;<)$ and $(<)_0^\infty \in \langle \bA \rangle$. \checked{Let $f \colon \Q^2 \to \Q$ be an injective operation.} Suppose that $\checked{f} \notin \fPol(\bA)$ and $\checked{f} \in \Pol(\Q; \langle \bA \rangle_0^\infty)$. Then $\Cyc \in \langle \bA \rangle$. 
\end{lemma}

\begin{proof}
Let $R$ be a~valued relation of $\bA$ of arity $k$ that is not improved by $\checked{f}$. Then there exist $s,t \in \Q^k$ such that
\[R(s)+R(t) < 2 R(\checked{f}(s,t)).\]
In particular, $R(s), R(t) < \infty$. Since $\Feas(R) \in \langle \bA \rangle_0^\infty$ is improved by $\checked{f}$, we have $R(\checked{f}(s,t)) < \infty$.
Let $u := \checked{f}(s,t)$. 
Note that we must have $R(s) < R(u)$ or $R(t) < R(u)$. Moreover, $E_u = E_s \cap E_t$. Let $v \in \{s,t\}$ be such that $R(v) < R(u) < \infty$. Note that we have $E_u \subseteq E_v$.
Let $O$ be a~maximal subset of $O_u$ such that there exists $w \in \Q^k$ satisfying
\begin{itemize}
    \item $R(w) \leq R(v)$, 
    \item $E_u \subseteq E_w$, and
    \item $O \subseteq O_w$,
\end{itemize}
and let $w$ be any such witness for $O$. Such a~maximal set $O$ must exist, because $v$ satisfies these conditions for $O = \emptyset$.

Since $R(w) \neq R(u)$ and $E_u \subseteq E_{w}$, there exist $i,j \in [k]$ such that $w_i \leq w_j$ and $u_i>u_j$. Without loss of generality me may assume $(i,j)=(1,2)$, because otherwise we permute the entries of $R$. Let $\phi$ be the following expression with the free variables $x_1,x_2$.  
\begin{equation} \label{eq:no-lex}
\min_{x_3, \dots, x_k} \left( R(x_1, \dots, x_k) + \sum_{(p,q) \in E_u} (=)_0^\infty(x_p,x_q) + \sum_{(p,q) \in O} (<)_0^\infty(x_p,x_q) \right) 
\end{equation}
Let $a,b \in \Q$ such that $a<b$. Then $\phi(b,a) = \phi(u_1, u_2) \leq R(u)$, because $O \subseteq O_u$. Suppose that $\phi(b,a) \leq R(w)$. Then there exists $w' \in \Q^k$
such that $w_1'>w'_2$ and $w_3', \dots, w_k'$ realize the minimum for $\phi(b,a)$ in \eqref{eq:no-lex} 
and hence $\phi(b,a)=R(w')\leq R(w)\leq R(v)$.
In particular, the sums in \eqref{eq:no-lex} are finite. Therefore, $O \cup \{(2,1)\} \subseteq O_{w'}$ and $E_u \subseteq E_{w'}$. Since $(2,1) \in O_{u} \setminus O$, this contradicts the choice of $O$ and $w$. Therefore, $\phi(b,a) > R(w)$.
Note that $\phi(w_1, w_2) \leq R(w)$. If $w_1 < w_2$, then $\phi(a,b) \leq R(w)$ and $\phi$ expresses $R_{\alpha, \beta, \gamma}$ where $\beta = \phi(a,b)$ and $\gamma=\phi(b,a)$. In particular, $\beta < \gamma < \infty$.  Therefore, by Lemma~\ref{lem:3-val-rels},
$\Cyc \in \langle \bA \rangle$.
Otherwise we have $w_1 = w_2$. Then $\phi(a,a) \leq R(w)$ and $\phi$ expresses $R_{\alpha, \beta, \gamma}$ where $\alpha = \phi(a,a) \leq R(w) < \phi(b,a) = \gamma$. If $\beta \geq \gamma$, then $\alpha < \min(\beta, \gamma)$, and otherwise $\beta < \gamma < \infty$. In both cases, $\Cyc \in \langle \bA \rangle$ by Lemma~\ref{lem:3-val-rels}.
\end{proof}

\subsection{Classification}

We can now state and prove the complexity dichotomy for temporal VCSPs. We first phrase the classification with 4 cases, where we distinguish between the tractable cases that are based on different algorithms. As a~next step, we formulate two corollaries each of which provides two concise mutually disjoint conditions that correspond to NP-completeness and polynomial-time tractability, respectively.

\begin{theorem}\label{thm:tvcsp-clas}
Let $\bA$ be a~valued structure
such that $\Aut(\Q; <) \subseteq \Aut(\bA)$. Then at least one of the following holds:
\begin{enumerate}
    \item $\langle \bA \rangle$ contains one of the relations $\Betw$, $\Cyc$, $\Sep$, $T_3$ (see Definition~\ref{def:rels}), $-T_3$, or $\Dis$ (see~\eqref{eq:D}). In this case, $\bA$ has  a reduct $\bA'$ over a finite signature such that $\VCSP(\bA')$ is NP-complete.
    \item $\const \in \fPol(\bA)$.
    \item \checked{$\lele \in \fPol(\bA)$ or $\lele^* \in \fPol(\bA)$.}
    \item $\pi_1^2 \in \fPol(\bA)$ and $\fPol(\bA)$ contains \checked{$\min$, $\mx$, $\mi$}, or one of their duals.
\end{enumerate}
In cases 2--4, for every reduct $\bA'$ of $\bA$ over a finite signature, $\VCSP(\bA')$ is in P.
\checked{Case 1 is mutually exclusive with each of the cases 2--4.}
\end{theorem}

\begin{proof}
Note that for every reduct $\bA'$ of $\bA$, the automorphism group $\Aut(\bA')$ contains $\Aut(\bA)$ and hence is oligomorphic. If $\langle \bA \rangle$ contains one of the relations $\Betw$, $\Cyc$, $\Sep$, $T_3$, $-T_3$, or $\Dis$,
then there is a reduct $\bA'$ of $\bA$ over a finite signature such that $\VCSP(\bA')$ is NP-hard by Lemma~\ref{lem:expr-reduce} and
Theorem~\ref{thm:tcsp}.
By Theorem~\ref{thm:NP}, $\VCSP(\bA')$ is in NP, therefore it is NP-complete. If $\const \in \fPol(\bA)$, then $\const \in \fPol(\bA')$ for every reduct $\bA'$ of $\bA$ over a finite signature, and $\VCSP(\bA')$ is in P by Lemma~\ref{lem:constant-pol}. Suppose therefore that $\const \notin \fPol(\bA)$ and
that $\langle \bA \rangle$ does not contain any of the relations $\Betw$, $\Cyc$, $\Sep$, $T_3$, $-T_3$, or $\Dis$. 

Let $\bA'$ be a reduct of $\bA$ with a finite signature.
If $\Aut(\bA) = \Sym(\Q)$, then by Theorem~\ref{thm:eqvcsp}, $\inj \in \fPol(\bA) \subseteq \fPol(\bA')$ and $\VCSP(\bA')$ is in P. 
\checked{Note that for $k \in \N$ and any $s, t \in A^k$, the tuple $\inj(s,t)$ lies in the same orbit as $\lele(s,t)$ under the action of $\Aut(\bA)$, hence $\inj \in \fPol(\bA)$ implies $\lele \in \fPol(\bA)$}
and therefore item 3 is satisfied. 

Suppose now that $\Aut(\bA) \neq \Sym(\Q)$. By Lemma~\ref{lem:expr-ord} we have $(<)_0^\infty \in \langle \bA \rangle$, and hence $\Aut(\bA)=\Aut(\Q;<)$. By Lemma~\ref{lem:lex-viol-crisp}  we have that $\bA$ is essentially crisp, \checked{ 
$\lele \in \Pol(\Q; \langle \bA \rangle_0^\infty)$ or $\lele^* \in \Pol(\Q; \langle \bA \rangle_0^\infty)$.} If $\bA$ is not essentially crisp, \checked{by Lemma~\ref{lem:lex-viol-soft} we have that $\lele$ or $\lele^*$ lies in} $\fPol(\bA) \subseteq \fPol(\bA')$.
Then $\VCSP(\bA')$ is in P by Lemma~\ref{lem:lex-algo} and item 3 holds. 

Finally, suppose that $\bA$ is essentially crisp. Then by Lemma~\ref{lem:ess-crisp} we have $\pi_1^2 \in \fPol(\bA)$. 
Since $\const \not\in \fPol(\bA)$, we have 
$\const \not\in \Pol(\Feas(\bA))$ (see Remark~\ref{rem:ess-crisp-fPol}).
Since $\langle \bA \rangle$ does not contain any of the relations $\Betw$, $\Cyc$, $\Sep$, $T_3$, $-T_3$, or $\Dis$, none of these relations are primitively positively definable in 
$\Feas(\bA)$.
By Theorem~\ref{thm:tcsp}, $\Pol(\Feas(\bA)) \subseteq \Pol(\Feas(\bA'))$ contains $\min$, $\mx$, $\mi$, $\lele$, or one of their duals and $\CSP(\Feas(\bA'))$ is in P. 
Thus $\VCSP(\bA')$ is in P. By Remark~\ref{rem:ess-crisp-fPol}, $\fPol(\bA)$ contains  $\min$, $\mx$, $\mi$, $\lele$, or one of their duals. Therefore, item 4 holds.

\checked{To see that case 1 is mutually exclusive with cases 2--4, observe that none of the relations listed in item 1 is improved by any of the fractional polymorphisms $\const$, $\min$, $\mx$, $\mi$, $\lele$, or one of their duals, which proves the claim by Lemma~\ref{lem:easy-Imp-fPol}.}
\end{proof}

Recall from Section~\ref{sect:aut} that a~valued structure $\bA$ with $\Aut(\Q; <) \subseteq \Aut(\bA)$ has a~(quantifier-free) first-order definition in $\Aut(\Q; <)$ with the defining formulas being disjunctions of conjunctions of atomic formulas over $(\Q;<)$.
Using this representation of $\bA$, we obtain the decidability of the complexity dichotomy from Theorem~\ref{thm:tvcsp-clas}. 

\begin{remark}
    We also obtain decidability if arbitrary first-order formulas may be used for defining the valued relations, because every first-order formula can be effectively transformed into such a formula. This holds more generally over so-called  \emph{finitely bounded homogeneous} structures; see, e.g., \cite[Proposition 7]{RydvalThesis}. Without the finite boundedness assumption, the problem can become undecidable~\cite{BPT-decidability-of-definability}. 
\end{remark}
 
\begin{proposition}\label{prop:dec}
Given a~first-order definition of a~valued structure $\bA$ with a~finite signature
in $(\Q;<)$, 
it is decidable whether \checked{item 1 in Theorem~\ref{thm:tvcsp-clas} holds, in particular, whether} $\VCSP(\bA)$ is in P or NP-complete.
\end{proposition}

\begin{proof}
\checked{Since by Theorem~\ref{thm:tvcsp-clas}, item 1 in Theorem~\ref{thm:tvcsp-clas} is mutually exclusive with any of items 2--4, it is enough to show that it is decidable whether items 2--4 hold.}
Since $\bA$ has a~finite signature, we can decide whether $\const$ improves $\bA$, i.e., whether item 2 holds, by checking for each valued relation of $\bA$ whether the value of a constant tuple is minimal. Similarly, we can decide whether \checked{operations $\lele$ or $\lele^*$ improve} $\bA$: \checked{given a valued relation $R$ of $\bA$ of arity $k$, we need to choose finitely many tuples from each orbit of $A^k$ under the action of $\Aut(\bA)$ representing all distributions of values between non-positive and positive entries and test whether the relation is improved on pairs of these orbit representatives.} 
Therefore, we can decide whether item 3 holds. Finally, we can decide whether $\pi_1^2$ improves $\bA$, \checked{equivalently, whether} $\bA$ is essentially crisp by Lemma~\ref{lem:ess-crisp}. In this case $\fPol(\bA)$ contains $\min$, $\mx$, $\mi$, $\lele$, or one of their duals if and only if $\Pol(\Feas(\bA))$ does (\cref{rem:ess-crisp-fPol}), which can be decided by Theorem~\ref{thm:tcsp}.
It follows that we can decide whether \checked{any of} items 2--4 holds, which implies the statement.
\end{proof}

We reformulate Theorem~\ref{thm:tvcsp-clas} with two mutually exclusive cases that capture the respective complexities of the VCSPs \checked{and relate it to the complexity of reducts of $(A; \langle \bA \rangle_0^\infty)$.}

\begin{corollary}\label{cor:tvcsp-clas}
Let $\bA$ be a~valued structure with a~finite signature such that $\Aut(\Q; <) \subseteq \Aut(\bA)$. Then exactly one of the following holds.
\begin{enumerate}
\item $\langle \bA \rangle$ contains one of the relations $\Betw$, $\Cyc$, $\Sep$, $T_3$, $-T_3$, or $\Dis$. In this case, $\VCSP(\bA)$ is NP-complete.
\item $(\Q; \langle \bA \rangle_0^\infty)$ is preserved by one of the operations $\const$, $\min$, $\mx$, $\mi$, $\lele$, or one of their duals. In this case, $\VCSP(\bA)$ is in P.
\end{enumerate}
\end{corollary}

\begin{proof}
Let $\bA':=(\Q; \langle \bA \rangle_0^\infty)$.
Theorem~\ref{thm:tcsp} states that either 
$\bA'$ primitively positively defines one of the relations $\Betw$, $\Cyc$, $\Sep$, $T_3$, $-T_3$, $\Dis$, or $\Pol(\bA')$ contains $\const$, $\min$, $\mx$, $\mi$, $\lele$, or one of their duals. Clearly, $\bA'$ primitively positively defines a~relation $R$ if and only if $R \in \langle \bA \rangle_0^\infty$, which is the case if and only if $R \in \langle \bA \rangle$.

It remains to discuss the implications for the complexity of $\VCSP(\bA)$. 
If item 1 holds, then $\VCSP(\bA)$ is NP-complete by Theorem~\ref{thm:tvcsp-clas}.
On the other hand, if item 1 does not hold, one of the items 2--4 in Theorem~\ref{thm:tvcsp-clas} applies and $\VCSP(\bA)$ is in P. 
\end{proof}

Note that the corollary above implies that if $\Aut(\Q; <) \subseteq \Aut(\bA)$, then the complexity of $\VCSP(\bA)$ is up to polynomial-time reductions determined by the complexity of the crisp relations $\bA$ can \checked{pp-}express. Loosely speaking, the complexity of such a~VCSP is determined solely by the CSPs that can be encoded in this VCSP. We formulate an alternative and more concise variant of the previous result in the style of Theorem~\ref{thm:tcsp-pwnu}.

\begin{corollary}\label{cor:clas-refor}
Let $\bA$ be a~valued structure with a finite signature such that $\Aut(\Q; <) \subseteq \Aut(\bA)$. Then exactly one of the following holds.
\begin{enumerate}
\item $\bA$, \checked{or, equivalently, $(A; \langle \bA \rangle_0^\infty)$} pp-constructs $K_3$. In this case, $\VCSP(\bA)$ is NP-complete.
\item $\Pol(\Q; \langle \bA \rangle_0^\infty)$ contains a~pwnu operation. In this case, $\VCSP(\bA)$ is in P.
\end{enumerate}
\end{corollary}

\begin{proof}
Let $\bA':=(\Q;\langle \bA \rangle_0^\infty)$. By Theorem~\ref{thm:NP}, $\VCSP(\bA)$ is in NP.
By Proposition~\ref{prop:pp-constr-rel}, $\bA$ pp-constructs $K_3$ if and only if $\bA'$ pp-constructs $K_3$ and in this case, $\VCSP(\bA)$ is NP-complete by Lemma~\ref{lem:hard}. Hence, it follows from Theorem~\ref{thm:tcsp-pwnu} applied to  $\bA'$ that either item 1 holds or $\Pol(\Q; \langle \bA \rangle_0^\infty)$ contains a~pwnu operation. Hence, if $\Pol(\Q; \langle \bA \rangle_0^\infty)$ contains a~pwnu operation, then $\bA$ does not pp-construct $K_3$. By Proposition~\ref{prop:hard} and Theorem~\ref{thm:crisp-hard},
$\Betw$, $\Cyc$, $\Sep$, $T_3$, $-T_3$, $\Dis \notin \langle \bA \rangle$ and therefore, item 2 from Corollary~\ref{cor:tvcsp-clas} applies and $\VCSP(\bA)$ is in P.
\end{proof}

Conjecture 9.3 in \cite{Resilience-VCSPs} states that, under some structural assumptions on $\bA$, $\VCSP(\bA)$ is in P whenever $\bA$ does not pp-construct $K_3$
(and is NP-hard otherwise)\footnote{The original formulation uses the structure $(\{0,1\};\OIT)$, but it is well-known that this structure pp-constructs $K_3$ and vice versa \cite{Book}.}. All temporal \checked{valued} structures
satisfy the assumptions of the conjecture and hence Corollary~\ref{cor:clas-refor} confirms the conjecture for the class of temporal VCSPs.

\section{Future work}
\label{sect:gvcsps}
In analogy to the development of the results on infinite-domain CSPs, we propose the class of valued structures that are preserved by all automorphisms of the countable random graph as a natural next step in the complexity classification of VCSPs on infinite domains.

\begin{question}\label{ques:graph}
Does the class of VCSPs of all valued structures $\bA$ over a finite signature such that $\Aut(\bA)$ contains the automorphism group of the countable random graph exhibit a P vs.\ NP-complete dichotomy? In particular, is $\VCSP(\bA)$ in P whenever $\bA$ does not pp-construct $K_3$?
\end{question}
A positive answer to the second question in Question~\ref{ques:graph} would confirm \cite[Conjecture 9.3]{Resilience-VCSPs} for valued structures preserved by all automorphisms of the countable random graph.

\section*{Acknowledgements}
Manuel Bodirsky and \v{Z}aneta Semani\v{s}inov\'{a} have been funded by the DFG (Project FinHom, Grant 467967530) and by the European Research Council (Project POCOCOP, ERC Synergy Grant 101071674). Views and opinions expressed are however those of the authors only and do not necessarily reflect those of the European Union or the European Research Council Executive Agency. Neither the European Union nor the granting authority can be held responsible for them.

\'Edouard Bonnet was supported by the ANR project TWIN-WIDTH (ANR-21-CE48-0014-01).

\v{Z}aneta Semani\v{s}inov\'{a} was funded in whole or in part by the Austrian Science Fund (FWF)10.55776/ESP6949724.





\bibliographystyle{elsarticle-num} 
\bibliography{global}

@STRING{lncs = {Lecture Notes in Computer Science} }

@STRING{LICS = {Proceedings of the Annual Symposium on Logic in Computer Science (LICS)} }

@STRING{STOC = {Proceedings of the Annual Symposium on Theory of Computing (STOC)} }

@STRING{ICALP = {Proceedings of the International Colloquium on Automata, Languages and Programming (ICALP)} }

@preamble{"\def\cprime{$'$} "}

@InProceedings{AnOrderOutOf,
  author =	{Mottet, Antoine and Nagy, Tom\'{a}\v{s} and Pinsker, Michael},
  title =	{{An Order out of Nowhere: A New Algorithm for Infinite-Domain {CSP}s}},
  booktitle =	{51st International Colloquium on Automata, Languages, and Programming (ICALP 2024)},
  pages =	{148:1--148:18},
  series =	{Leibniz International Proceedings in Informatics (LIPIcs)},
  ISBN =	{978-3-95977-322-5},
  ISSN =	{1868-8969},
  year =	{2024},
  volume =	{297},
  publisher =	{Schloss Dagstuhl -- Leibniz-Zentrum f{\"u}r Informatik},
  address =	{Dagstuhl, Germany},
  URL =		{https://drops.dagstuhl.de/entities/document/10.4230/LIPIcs.ICALP.2024.148},
  URN =		{urn:nbn:de:0030-drops-202912},
  doi =		{10.4230/LIPIcs.ICALP.2024.148},
}

@inproceedings{ChatziafratisM23,
  author       = {Vaggos Chatziafratis and
                  Konstantin Makarychev},
  title        = {Triplet Reconstruction and all other Phylogenetic {CSP}s are Approximation
                  Resistant},
  booktitle    = {64th {IEEE} Annual Symposium on Foundations of Computer Science, {FOCS}
                  2023, Santa Cruz, CA, USA, November 6-9, 2023},
  pages        = {253--284},
  publisher    = {{IEEE}},
  year         = {2023},
  url          = {https://doi.org/10.1109/FOCS57990.2023.00024},
  doi          = {10.1109/FOCS57990.2023.00024},
  bibsource    = {dblp computer science bibliography, https://dblp.org}
}

@article{Zhuk20,
  author    = {Dmitriy Zhuk},
  title     = {A Proof of the {CSP} Dichotomy Conjecture},
  journal   = {J. {ACM}},
  volume    = {67},
  number    = {5},
  pages     = {30:1--30:78},
  year      = {2020},
  url       = {https://doi.org/10.1145/3402029},
  doi       = {10.1145/3402029}
}

@inproceedings{ZhukFVConjecture,
  author    = {Dmitriy Zhuk},
  title     = {A Proof of {CSP} Dichotomy Conjecture},
  booktitle = {58th {IEEE} Annual Symposium on Foundations of Computer Science, {FOCS}
               2017, {B}erkeley, {CA}, {USA}, {O}ctober 15-17},
  pages     = {331-342},
  year      = {2017},
  note = {https://arxiv.org/abs/1704.01914.}
}

@inproceedings{BulatovFVConjecture,
  author    = {Andrei A. Bulatov},
  title     = {A Dichotomy Theorem for Nonuniform {CSP}s},
  booktitle = {58th {IEEE} Annual Symposium on Foundations of Computer Science, {FOCS}
               2017, {B}erkeley, {CA}, {USA}, {O}ctober 15-17},
  pages     = {319-330},
  year      = {2017}
}

@preamble{
   "\def\cprime{$'$} "
}

@inproceedings{CohenCooperJeavonsVCSP,
  author    = {David A. Cohen and
               Martin C. Cooper and
               Peter Jeavons},
  title     = {An Algebraic Characterisation of Complexity for Valued Constraints},
  booktitle = {Principles and Practice of Constraint Programming - {CP} 2006, 12th
               International Conference, {CP} 2006, Nantes, France, September 25-29,
               2006, Proceedings},
  pages     = {107--121},
  year      = {2006},
  url       = {https://doi.org/10.1007/11889205\_10},
  doi       = {10.1007/11889205\_10},
  bibsource = {dblp computer science bibliography, https://dblp.org}
}

@article{KolmogorovKR17,
  author    = {Vladimir Kolmogorov and
               Andrei A. Krokhin and
               Michal Rol{\'{\i}}nek},
  title     = {The Complexity of General-Valued {CSP}s},
  journal   = {{SIAM} J. Comput.},
  volume    = {46},
  number    = {3},
  pages     = {1087--1110},
  year      = {2017},
  url       = {https://doi.org/10.1137/16M1091836},
  doi       = {10.1137/16M1091836},
  bibsource = {dblp computer science bibliography, https://dblp.org}
}

@incollection{Pol,
  author    = {Libor Barto and
               Andrei A. Krokhin and
               Ross Willard},
  title     = {Polymorphisms, and How to Use Them},
  booktitle = {The Constraint Satisfaction Problem: Complexity and Approximability},
  publisher = {Schloss Dagstuhl - Leibniz-Zentrum fuer Informatik},
  pages     = {1-44},
  year      = {2017}
}

@article{CorrelationClustering,
author = {Nikhil Bansal and Avrim Blum and Shuchi Chawla},
title = {Correlation clustering},
journal = {Machine Learning}, 
volume = {56},
number = {1-3},
pages = {89-113},
year = {2004}}

@article{Multicut,
  author       = {Nicolas Bousquet and
                  Jean Daligault and
                  St{\'{e}}phan Thomass{\'{e}}},
  title        = {Multicut Is {FPT}},
  journal      = {{SIAM} J. Comput.},
  volume       = {47},
  number       = {1},
  pages        = {166--207},
  year         = {2018},
  url          = {https://doi.org/10.1137/140961808},
  doi          = {10.1137/140961808},
  bibsource    = {dblp computer science bibliography, https://dblp.org}
}

@Article{wonderland,
author={Libor Barto and Jakub Opr\v{s}al and Michael Pinsker},
title={The wonderland of reflections},
journal = {Israel Journal of Mathematics},
volume=223,
number=1,
year=2018,
pages={363-398}
}

@article{Products,
author = {Bodirsky, Manuel and Jonsson, Peter and Martin, Barnaby and Mottet, Antoine and Semani\v{s}inov\'{a}, {\v{Z}}aneta},
title = {Complexity Classification Transfer for {CSP}s via Algebraic Products},
journal = {SIAM Journal on Computing},
volume = {53},
number = {5},
pages = {1293-1353},
year = {2024},
doi = {10.1137/22M1534304}
}

@article{Interval,
  author       = {Konrad K. Dabrowski and
                  Peter Jonsson and
                  Sebastian Ordyniak and
                  George Osipov and
                  Marcin Pilipczuk and
                  Roohani Sharma},
  title        = {Parameterized Complexity Classification for Interval Constraints},
  journal      = {CoRR},
  volume       = {abs/2305.13889},
  year         = {2023},
  url          = {https://doi.org/10.48550/arXiv.2305.13889},
  doi          = {10.48550/ARXIV.2305.13889},
  eprinttype    = {arXiv},
  eprint       = {2305.13889},
  bibsource    = {dblp computer science bibliography, https://dblp.org}
}

@article{OsipovPilipczuk,
  author       = {George Osipov and
                  Marcin Pilipczuk},
  title        = {Directed Symmetric Multicut is {W}[1]-hard},
  journal      = {CoRR},
  volume       = {abs/2310.05839},
  year         = {2023},
  url          = {https://doi.org/10.48550/arXiv.2310.05839},
  doi          = {10.48550/ARXIV.2310.05839},
  eprinttype    = {arXiv},
  eprint       = {2310.05839},
  bibsource    = {dblp computer science bibliography, https://dblp.org}
}

@article{Resilience,
  author    = {Cibele Freire and
               Wolfgang Gatterbauer and
               Neil Immerman and
               Alexandra Meliou},
  title     = {The Complexity of Resilience and Responsibility for Self-Join-Free
               Conjunctive Queries},
  journal   = {Proc. {VLDB} Endow.},
  volume    = {9},
  number    = {3},
  pages     = {180--191},
  year      = {2015},
  url       = {http://www.vldb.org/pvldb/vol9/p180-freire.pdf},
  doi       = {10.14778/2850583.2850592},
  bibsource = {dblp computer science bibliography, https://dblp.org}
}

@inproceedings{NewResilience,
  author    = {Cibele Freire and
               Wolfgang Gatterbauer and
               Neil Immerman and
               Alexandra Meliou},
  title     = {New Results for the Complexity of Resilience for Binary Conjunctive
               Queries with Self-Joins},
  booktitle = {Proceedings of the 39th {ACM} {SIGMOD-SIGACT-SIGAI} Symposium on Principles
               of Database Systems, {PODS} 2020, Portland, OR, USA, June 14-19, 2020},
  pages     = {271--284},
  year      = {2020},
  doi       = {10.1145/3375395.3387647},
  bibsource = {dblp computer science bibliography, https://dblp.org}
}

@article{LatestResilience,
author = {Makhija, Neha and Gatterbauer, Wolfgang},
title = {A Unified Approach for Resilience and Causal Responsibility with Integer Linear Programming ({ILP}) and {LP} Relaxations},
year = {2023},
issue_date = {December 2023},
publisher = {Association for Computing Machinery},
address = {New York, NY, USA},
volume = {1},
number = {4},
url = {https://doi.org/10.1145/3626715},
doi = {10.1145/3626715},
journal = {Proc. ACM Manag. Data},
month = {dec},
articleno = {228},
numpages = {27},
}

@inproceedings{OsipovW23,
  author       = {George Osipov and
                  Magnus Wahlstr{\"{o}}m},
  editor       = {Inge Li G{\o}rtz and
                  Martin Farach{-}Colton and
                  Simon J. Puglisi and
                  Grzegorz Herman},
  title        = {Parameterized Complexity of Equality MinCSP},
  booktitle    = {31st Annual European Symposium on Algorithms, {ESA} 2023, September
                  4-6, 2023, Amsterdam, The Netherlands},
  series       = {LIPIcs},
  volume       = {274},
  pages        = {86:1--86:17},
  publisher    = {Schloss Dagstuhl - Leibniz-Zentrum f{\"{u}}r Informatik},
  year         = {2023},
  url          = {https://doi.org/10.4230/LIPIcs.ESA.2023.86},
  doi          = {10.4230/LIPICS.ESA.2023.86},
  bibsource    = {dblp computer science bibliography, https://dblp.org}
}

@inproceedings{EibenRW22,
  author       = {Eduard Eiben and
                  Cl{\'{e}}ment Rambaud and
                  Magnus Wahlstr{\"{o}}m},
  editor       = {Holger Dell and
                  Jesper Nederlof},
  title        = {On the Parameterized Complexity of Symmetric Directed Multicut},
  booktitle    = {17th International Symposium on Parameterized and Exact Computation,
                  {IPEC} 2022, September 7-9, 2022, Potsdam, Germany},
  series       = {LIPIcs},
  volume       = {249},
  pages        = {11:1--11:17},
  publisher    = {Schloss Dagstuhl - Leibniz-Zentrum f{\"{u}}r Informatik},
  year         = {2022},
  url          = {https://doi.org/10.4230/LIPIcs.IPEC.2022.11},
  doi          = {10.4230/LIPICS.IPEC.2022.11},
  bibsource    = {dblp computer science bibliography, https://dblp.org}
}

@Article{ApproxOrderingCSP,
title = {Beating the random ordering is hard: Every ordering {CSP} is approximation resistant}, 
journal = {SIAM Journal on Computing}, 
year = {2011}, 
volume = {40},
number = {3},
pages = {878-914},
author = {Venkatesan Guruswami and Johan H\r{a}stad and Rajsekar Manokaran and Prasad Raghavendra and Moses Charikar}}

@BOOK{Oligo,
author = {Peter J. Cameron},
title = {Oligomorphic permutation groups},
publisher = {Cambridge University Press},
address = {Cambridge},
year = {1990}}

@article{VCSP-Galois,
  author    = {David A. Cohen and
               Martin C. Cooper and
               P{\'{a}}id{\'{\i}} Creed and
               Peter G. Jeavons and
               Stanislav \v{Z}ivn\'y},
  title     = {An Algebraic Theory of Complexity for Discrete Optimization},
  journal   = {{SIAM} J. Comput.},
  volume    = {42},
  number    = {5},
  pages     = {1915-1939},
  year      = {2013}
}

@article{FullaZivny,
  author    = {Peter Fulla and
               Stanislav \v{Z}ivn\'y},
  title     = {A {G}alois Connection for Weighted (Relational) Clones of Infinite Size},
  journal   = {{TOCT}},
  volume    = {8},
  number    = {3},
  pages     = {9:1-9:21},
  year      = {2016}
}

@inproceedings{KozikOchremiak15,
  author    = {Marcin Kozik and
               Joanna Ochremiak},
  title     = {Algebraic Properties of Valued Constraint Satisfaction Problem},
  booktitle = {Automata, Languages, and Programming - 42nd International Colloquium,
               {ICALP} 2015, Kyoto, Japan, July 6-10, 2015, Proceedings, Part {I}},
  pages     = {846-858},
  year      = {2015}
}

@inproceedings{ThapperZivny13,
  author    = {Johan Thapper and
               Stanislav \v{Z}ivn\'y},
  title     = {The complexity of finite-valued {CSP}s},
  booktitle = {Proceedings of the Symposium on Theory of Computing Conference ({STOC}), Palo Alto, CA,
               USA, June 1-4, 2013},
  pages     = {695-704},
  year      = {2013}
}

@article{cohen2006complexity,
  title={The complexity of soft constraint satisfaction},
  author={Cohen, David A and Cooper, Martin C and Jeavons, Peter G and Krokhin, Andrei A},
  journal={Artificial Intelligence},
  volume={170},
  number={11},
  pages={983-1016},
  year={2006},
  publisher={Elsevier}
}

@article{RydvalDescr,
author = {Bodirsky, Manuel and Rydval, Jakub},
title = {On the Descriptive Complexity of Temporal Constraint Satisfaction Problems},
year = {2022},
issue_date = {February 2023},
publisher = {Association for Computing Machinery},
address = {New York, NY, USA},
volume = {70},
number = {1},
issn = {0004-5411},
url = {https://doi.org/10.1145/3566051},
doi = {10.1145/3566051},
journal = {J. ACM},
month = {dec},
articleno = {2},
numpages = {58}
}

@article{BBSpatial, 
author = {Bodirsky, Manuel and Bodor, Bertalan}, title = {A Complexity Dichotomy in Spatial Reasoning via {R}amsey Theory}, year = {2024}, issue_date = {June 2024}, publisher = {Association for Computing Machinery}, address = {New York, NY, USA}, volume = {16}, number = {2}, issn = {1942-3454}, url = {https://doi.org/10.1145/3649445}, doi = {10.1145/3649445}, journal = {ACM Trans. Comput. Theory}, month = {jun}, articleno = {10}, numpages = {39} }

@phdthesis{RydvalThesis,
  author    = {Jakub Rydval},
  title     = {Using Model Theory to Find Decidable and Tractable Description Logics
               with Concrete Domains},
  school    = {Dresden University of Technology, Germany},
  year      = {2022},
  url       = {https://nbn-resolving.org/urn:nbn:de:bsz:14-qucosa2-799074},
  urn       = {urn:nbn:de:bsz:14-qucosa2-799074},
  bibsource = {dblp computer science bibliography, https://dblp.org}
}

@article{BodirskyGreinerRydval,
  TITLE = {{Tractable Combinations of Temporal CSPs}},
  AUTHOR = {Manuel Bodirsky and Johannes Greiner and Jakub Rydval},
  URL = {https://lmcs.episciences.org/9609},
  DOI = {10.46298/lmcs-18(2:11)2022},
  JOURNAL = {{Logical Methods in Computer Science}},
  VOLUME = {{Volume 18, Issue 2}},
  YEAR = {2022},
  MONTH = May,
}

@article{BodPin-Schaefer-both,
  author = {Manuel Bodirsky and Michael Pinsker},
title = {Schaefer's theorem for graphs},  
journal = {Journal of the ACM},
volume=62,
number=3,
  doi       = {10.1145/2764899},
year = {2015}, 
pages={52 pages (article number 19)},
note={A conference version appeared in the Proceedings of STOC 2011, pages 655-664}
}

@article{BPT-decidability-of-definability,
  author = {Manuel Bodirsky and Michael Pinsker and Todor Tsankov},
  title = {Decidability of definability},
  journal = {Journal of Symbolic Logic},
  volume=78,
  number=4,
  pages={1036-1054},
  year = {2013},
  note={A conference version appeared in the Proceedings of
the Twenty-Sixth Annual IEEE Symposium on. Logic in Computer Science (LICS 2011), pages {321-328}}
}

@article{phylo-long,
  author    = {Manuel Bodirsky and Jens K. Mueller},
  title     = {Rooted Phylogeny Problems},
  journal = {Logical Methods in Computer Science},
  volume = {7},
  number = {4},
  note =  {An extended abstract appeared in the proceedings of {ICDT'10}},
  year      = {2011}}

@Article{tcsps-journal,
author = {Manuel Bodirsky and Jan K\'ara},
title = {The Complexity of Temporal Constraint Satisfaction Problems},
journal = {Journal of the ACM},
volume = {57},
number = {2},
pages = {1-41},
  doi       = {10.1145/1667053.1667058},
note = {An extended abstract appeared in the Proceedings of the Symposium on Theory of Computing (STOC)}, 
year = {2009}}

@Article{ecsps,
  author = 	 {Manuel Bodirsky and Jan K\'ara},
  title = 	 {The Complexity of Equality Constraint Languages},
  journal = {Theory of Computing Systems},
  volume = {3},
  number = {2}, 
    doi       = {10.1007/s00224-007-9083-9},
  pages = {136-158},
  note = {A conference version appeared in the proceedings of Computer Science Russia {(CSR'06)}},
  year = 	 {2008},
}

@Article{BodirskyNesetrilJLC,
author = {Manuel Bodirsky and Jaroslav Ne\v{s}et\v{r}il},
title  = {Constraint Satisfaction with Countable Homogeneous Templates},
journal = {Journal of Logic and Computation},
volume = {16},
number = {3},
doi       = {10.1093/logcom/exi083},
pages = {359-373},
year   = {2006}}

@Article{ll,
author = {Manuel Bodirsky and Jan K\'ara},
title = {A fast algorithm and {D}atalog Inexpressibility for temporal reasoning},
journal = {ACM Transactions on Computational Logic}, 
volume = {11},
number = {3},
year = {2010}}

@Book{Book,
author = {Manuel Bodirsky},
title  = {Complexity of Infinite-Domain Constraint Satisfaction},
year   = {2021},
doi = {10.1017/9781107337534}, 
publisher = {Cambridge University Press},
series = {Lecture Notes in Logic (52)}}

@InProceedings{BodirskyGrohe,
   author = {Manuel Bodirsky and Martin Grohe},
   booktitle = ICALP,
   title = {Non-dichotomies in Constraint Satisfaction Complexity},
   year = {2008},
   month = {July},
   publisher = {Springer Verlag},
   series = LNCS,
   editor = {Luca Aceto and Ivan Damgard and Leslie Ann Goldberg and Magn\'us M. Halld\'orsson and Anna Ing\'olfsd\'ottir and Igor Walukiewicz},
   pages = {184 -196}
}

@article{ViolaZivny,
  author       = {Caterina Viola and
                  Stanislav \v{Z}ivn{\'{y}}},
  title        = {The Combined Basic {LP} and Affine {IP} Relaxation for Promise {VCSP}s
                  on Infinite Domains},
  journal      = {{ACM} Trans. Algorithms},
  volume       = {17},
  number       = {3},
  pages        = {21:1--21:23},
  year         = {2021},
  url          = {https://doi.org/10.1145/3458041},
  doi          = {10.1145/3458041},
  bibsource    = {dblp computer science bibliography, https://dblp.org}
}

@article{SchneiderViola,
title = {An application of {F}arkas' lemma to finite-valued constraint satisfaction problems over infinite domains},
journal = {Journal of Mathematical Analysis and Applications},
volume = {517},
number = {1},
pages = {126591},
year = {2023},
issn = {0022-247X},
doi = {https://doi.org/10.1016/j.jmaa.2022.126591},
url = {https://www.sciencedirect.com/science/article/pii/S0022247X22006059},
author = {Friedrich Martin Schneider and Caterina Viola},
}

@phdthesis{ViolaThesis,
author={Caterina Viola},
title={Valued Constraint Satisfaction Problems over Infinite Domains},
school={TU Dresden},
year=2020
}

@article{BodirskyMaminoViola-Journal,
    AUTHOR = {Manuel Bodirsky and Marcello Mamino and Caterina Viola},
     TITLE = {Piecewise Linear Valued {CSP}s Solvable by Linear Programming Relaxation},
     journal = {ACM Transactions of Computational Logic}, 
     volume = {23},
     number = {1},
     year = {2022}, 
     pages = {1-35}, 
     doi = {10.1145/3488721}, 
     note = {Preprint arXiv:1912.09298}
}

@article{Resilience-VCSPs-arxiv,
  author       = {Manuel Bodirsky and
                  \v{Z}aneta Semani\v{s}inov{\'{a}} and
                  Carsten Lutz},
  title        = {The Complexity of Resilience Problems via Valued Constraint Satisfaction
                  Problems},
  journal      = {CoRR},
  volume       = {abs/2309.15654},
  year         = {2023},
  url          = {https://doi.org/10.48550/arXiv.2309.15654},
  doi          = {10.48550/ARXIV.2309.15654},
  eprinttype    = {arXiv},
  eprint       = {2309.15654}
}

@inproceedings{Resilience-VCSPs,
author = {Bodirsky, Manuel and Semani\v{s}inov\'{a}, {\v{Z}}aneta and Lutz, Carsten},
title = {The Complexity of Resilience Problems via Valued Constraint Satisfaction Problems},
year = {2024},
isbn = {9798400706608},
publisher = {Association for Computing Machinery},
url = {https://doi.org/10.1145/3661814.3662071},
doi = {10.1145/3661814.3662071},
booktitle = {Proceedings of the 39th Annual ACM/IEEE Symposium on Logic in Computer Science},
articleno = {14},
numpages = {14},
location = {Tallinn, Estonia},
series = {LICS '24}
}

@InProceedings{equalityminCSP,
  author =	{Osipov, George and Wahlstr\"{o}m, Magnus},
  title =	{Parameterized Complexity of Equality {M}in{CSP}},
  booktitle =	{31st Annual European Symposium on Algorithms (ESA 2023)},
  pages =	{86:1--86:17},
  series =	{Leibniz International Proceedings in Informatics (LIPIcs)},
  ISBN =	{978-3-95977-295-2},
  ISSN =	{1868-8969},
  year =	{2023},
  volume =	{274},
  publisher =	{Schloss Dagstuhl -- Leibniz-Zentrum f{\"u}r Informatik},
  address =	{Dagstuhl, Germany},
  URL =		{https://drops-dev.dagstuhl.de/entities/document/10.4230/LIPIcs.ESA.2023.86},
  URN =		{urn:nbn:de:0030-drops-187393},
  doi =		{10.4230/LIPIcs.ESA.2023.86},
}

@inproceedings{PointAlgebraMinCSP,
title = "Parameterized Complexity of {M}in{CSP} over the {P}oint {A}lgebra",
author = "George Osipov and Marcin Pilipczuk and Magnus Wahlstr{\"o}m",
year = "2024",
month = jun,
day = "23",
language = "English",
booktitle = "Proceedings of ESA 2024",
}

@phdthesis{ThesisZaneta,
author = {\v{Z}aneta Semani\v{s}inov\'{a}},
title  = {{V}alued {C}onstraint {S}atisfaction in
{S}tructures with an {O}ligomorphic
{A}utomorphism {G}roup},
year   = {2025},
address      = {Dresden, Germany},
note         = {Available at \url{
https://nbn-resolving.org/urn:nbn:de:bsz:14-qucosa2-974737}},
school       = {TU Dresden},
type         = {Dissertation}
}






\end{document}